\documentclass[11pt]{article}
\usepackage[utf8]{inputenc}
\usepackage[english]{babel}
\usepackage{geometry}

\usepackage{enumitem}

\usepackage{amssymb}   
\usepackage{stmaryrd}
\usepackage{amsmath}
\usepackage{lipsum}
\usepackage{amsmath,amscd}
\usepackage{dsfont}
\usepackage{algorithm}
\usepackage[italicComments=false]{algpseudocodex}
\newcommand{\cmt}[1]{\Comment{\textcolor{gray!90!black}{#1}}}

\usepackage{pgf}  		
\usepackage{amsthm} 	
\usepackage{thmtools} 	

\usepackage{wrapfig}			
\usepackage{standalone}
\usepackage{tikz,pgfplots}
\usepackage{tikz-cd}
\pgfplotsset{compat=1.15}
\usepackage{mathrsfs}
\usetikzlibrary{arrows}
\usetikzlibrary{decorations.markings}
\usepackage[framemethod=tikz]{mdframed}
\usepackage{adjustbox}
\usepackage{subcaption}

\usepackage{caption}

\usepackage{hyperref}
\usepackage{xcolor}

\usepackage{enumitem}
\usepackage{blindtext}
\usepackage{comment}
\usepackage{listings}

\usepackage{mathtools}
\newcommand{\defeq}{\vcentcolon=}

\definecolor{shadegray}{gray}{0.96} 	

 \newtheorem {definition} {Definition} [subsection] 
  \newtheorem {theorem}[definition] {Theorem} 
  \newtheorem {corollary}[definition]  {Corollary}
\newtheorem {prop}[definition] {Proposition}
\newtheorem {lemma}[definition] {Lemma} 
\newtheorem {remark}[definition] {Remark} 
\newtheorem {example}[definition] {Example} 
\numberwithin{equation}{section}
\newcommand{\theoremcolorbox}[2]{\def\theoremframecommand{\fcolorbox{#1}{#2}}}
\newtheorem{hyp}{Hypothesis}
\theoremcolorbox{black}{cyan}

\newcommand{\encad}[1]{%
\fbox{\begin{minipage}[t]{0.92\linewidth}%
#1\end{minipage}}}

\newcommand*{\EnsembleQuotient}[2]%
{\ensuremath{%
    #1/\!\raisebox{-.65ex}{\ensuremath{#2}}}}
    
\newcommand*{\RZ}{\EnsembleQuotient{\mathbb{R}}{\mathbb{Z}}}

\newcommand*{\fonction}[5]{
\begin{align*}
#1 : \left\{\begin{array}{lcl}  #2 &\to      &#3\\
                                #4    &\mapsto  &#5
\end{array} \right. .      
\end{align*} 
}

\newcommand*{\fonctionv}[5]{
\begin{align*}
#1 : \left\{\begin{array}{lcl}  #2 &\to      &#3\\
                                #4    &\mapsto  &#5
\end{array} \right. ,      
\end{align*} 
}

\newcommand*{\ninf}[1]{
\lVert #1 \rVert_\infty
}

\newcommand*{\bninf}[1]{
\Big\lVert #1 \Big\rVert_\infty
}

\newcommand*{\dsph}[2]{\partial_s\varphi(#1,#2)}
\newcommand*{\dsphn}[3]{\partial_s\varphi^{(#3)}(#1,#2)}
\newcommand*{\dxph}[3]{\partial_x^{#3}\varphi(#1,#2)}
\newcommand*{\dxphn}[4]{\partial_x^{#4}\varphi^{(#3)}(#1,#2)}
\newcommand*{\dxsph}[4]{\partial_x^{#3}\partial_s^{#4}\varphi(#1,#2)}
\newcommand*{\ph}[2]{\varphi(#1,#2)}
\newcommand*{\phn}[3]{\varphi^{(#3)}(#1,#2)}

\newcommand*{\rz}{\mathbb{R}/\mathbb{Z}}

\newcommand{\N}{\mathbb{N}}
\newcommand{\R}{\mathbb{R}}

\newcommand{\uo}[3]{\underset{#1}{\overset{#2}{#3}}}

\title{Numerical bounds on the regularity of an invariant function: Probability of extinction of Galton-Watson processes in dynamical environments}
\author{Thomas Morand\thanks{Université Paris-Saclay, CNRS, Laboratoire de mathématiques d’Orsay, 91405, Orsay, France\\ Email: thomas.morand@universite-paris-saclay.fr}\\\textit{Université Paris-Saclay}}
\date{2025}

\begin{document}  

\renewcommand{\proofname}{Proof}
\renewcommand\refname{References}
\renewcommand*\contentsname{Table of contents}
\newcommand{\argmin}{\text{argmin}}
\maketitle

\begin{abstract}
We study the Lyapunov exponents of models that are close to skew product systems over a $\mathcal{C}^1$ uniformly expanding transformation of the circle. For a continuous fibre map $\varphi$, analytic, increasing, and convex in the fibre variable, we consider the smallest invariant function $q$ satisfying $q(x) = \varphi(x, q(Tx))$. We provide rigorous numerical bounds on two Lyapunov exponents (the fibre exponent and the base exponent), and present algorithms to compute these bounds effectively.

We then apply this framework to Galton-Watson processes in dynamical environments in the uniformly supercritical case. The probability of extinction $q$ of the process is the invariant function of the associated system. Using the previously computed Lyapunov exponents, we control the Hölder regularity and differentiability class of the probability of extinction.
\end{abstract}

\section{Introduction}

Understanding the regularity of invariant graphs is a central issue in the study of skew product systems \cite{MR1605989,MR1677161,MR1898800,MR3816739}. From a numerical point of view, their regularity can often be estimated by controlling Lyapunov exponents.
The numerical estimation of Lyapunov exponents has been investigated in several contexts, 
notably for random matrix products \cite{MR4014663,MR4304495}, 
for iterated function systems contracting on average \cite{MR1751564}, 
and for interval transformations \cite{MR2430654,MR4538289} 
(including maps with critical points, discontinuities, or unbounded derivatives).

In this article, we focus on a model close to skew products, arising from the study of 
Galton--Watson processes in dynamical environments in the uniformly supercritical case introduced in 
\cite{morand2024galtonwatsonprocessesdynamicalenvironments}. 
The invariant function of this model represents the probability of extinction. 
We express the Hölder regularity of this invariant function as a ratio between two Lyapunov exponents. 
The main numerical difficulty in controlling the regularity of the invariant graph lies in the fact that the Lyapunov exponent in the fibre direction depends on the invariant function itself.

\subsection*{Numerical bounds on the Lyapunov exponent in the fibre}

We consider a $\mathcal{C}^1$ uniformly expanding transformation of the circle $T:\RZ\to\RZ$, and a continuous application $\varphi:\RZ\times[0,1]\to[0,1]$. Moreover, we assume that for all $x\in\RZ$, the function $s\mapsto\varphi(x,s)$ is analytic on $[0,1)$, increasing, convex, and that $\varphi(x,1)=1$. We denote by $\partial_s\varphi$ the derivative of $\varphi$ with respect to the second variable and by $\partial_x\varphi$ the derivative of $\varphi$ with respect to the first variable if it is well defined. For all $n\in\N$, we define $\varphi^{(n)}:\RZ\times[0,1]\to[0,1]$ by induction: for all $x\in\RZ$ and $s\in[0,1]$, \begin{align*}
\varphi^{(n)}(x,s)\defeq
\left\{\begin{array}{ll}
    & s \text{ if } n=0,\\
    & \varphi(x,\varphi^{(n-1)}(Tx,s)) \text{ otherwise.}
    \end{array}
    \right.
    \end{align*}
We consider the solutions $f:\RZ\to[0,1]$ of the functional equation: \begin{align}\label{invgraph}
     f(x)=\varphi(x,f(Tx)).
\end{align}
The constant function equal to $1$ is a solution of \eqref{invgraph}. We denote by $q$ the smallest solution of this functional equation. This model is closely related to models of invariant graphs of skew product systems, as presented in references \cite{MR1605989,MR1677161,MR1898800,MR3170606,MR3816739}. See \cite[Subsection 1.4]{morand2024galtonwatsonprocessesdynamicalenvironments} for details of the connections. \begin{remark}
     If we consider an invertible transformation of the circle instead of an uniformly expanding transformation, we would recover exactly a skew product system by considering the dynamics in the past. Let $\Psi$ be defined by: \fonction{\Psi}{\RZ\times [0,1]}{\RZ\times [0,1]}{(x,s)}{(T^{-1}x,\varphi(T^{-1}x,s))}
Then, for all $x\in\RZ$, $n\in\mathbb{N}$, and $s\in[0,1]$, 
\begin{align*}
    \Psi^n(x,s)=(T^{-n}x,\varphi^{(n)}(T^{-n}x,s)).
\end{align*}
\end{remark}
We refer to $q$ as the invariant function of the system $(\varphi,T)$ by analogy with models of invariant graphs of skew product systems. 

In this article, we study two Lyapunov exponents: \begin{itemize}
    \item the Lyapunov exponent of the invariant function $q$ in the fibre direction:
  \begin{align}\label{lf}
      \lambda_F&\defeq\lim_{n\to\infty}\frac{1}{n}\sup_{x\in\rz}\log(\partial_s\varphi^{(n)}(x,q(T^n x))),
      \end{align}    
    \item the Lyapunov exponent of the transformation $T$:
  \begin{align}\label{lu}
    \lambda_u\defeq\lim_{n\to\infty}\frac{1}{n}\sup_{x\in\rz}\log(|(T^n)'(x)|). 
  \end{align}\end{itemize}
  
  We obtain rigorous numerical bounds on these Lyapunov exponents. The difficulty lies in the fact that the Lyapunov exponent in the fibre depends on the invariant function, which is not explicit and may have low and unknown regularity. According to the specification property \cite[Corollary~5]{MR718829}, to obtain a lower bound on the Lyapunov exponent in the fibre, it is sufficient to have a lower bound on the invariant function only along the periodic orbits Proposition~\ref{prop2.1}. We then estimate the periodic orbits of $T$ of length at most some $M\in\N$. Corollary~\ref{coro2.2} allows us to obtain an upper bound on the Lyapunov exponent in the fibre. We first need to obtain an upper bound on the invariant function using arithmetic intervals and Lemma~\ref{lem2.1}.
  
 It is easier to control the Lyapunov exponents of the transformation and previous work has been done in more complicated cases \cite{MR2430654,MR4538289}.
  
  All the code has been written in the Julia programming language and is available on GitHub\footnote{\href{https://github.com/thomas-morand/holder_regularity_of_invariant_graph}{Link to code: https://github.com/thomas-morand/holder\_regularity\_of\_invariant\_graph}}.
  
 \subsection*{Application to the probability of extinction of Galton-Watson processes in dynamical environments in the uniformly supercritical case}

We apply these results to study the model of Galton-Watson processes in dynamical environments defined in \cite{morand2024galtonwatsonprocessesdynamicalenvironments}. These are Galton-Watson processes (stochastic branching processes in which each individual in a generation produces a random number of offspring) in which the law of reproduction varies between generations according to a dynamical system. We consider a $\mathcal{C}^1$ uniformly expanding transformation of the circle $T:\RZ\to\RZ$ together with a continuous function $x\in\RZ\mapsto \mu_x\in\mathcal{P}(\mathbb{N})$, called the law of reproduction of the Galton-Watson process in dynamical environments (a function that associates a probability on $\N$ to each $x\in\RZ$). The law of reproduction of the Galton–Watson process in dynamical environments with initial environment $x\in\RZ$ at generation $n\in\mathbb{N}$ is $\mu_{T^nx}$, i.e., the law of reproduction between generations evolves under the action of $T$. For a fixed $x\in\RZ$, the process is a Galton-Watson process in varying environments, as studied in particular in \cite{MR0368197,MR0365733,MR2384553,MR4094390}. We denote by $(Z_n(x))_{n\in\N}$ the Galton–Watson process in dynamical environment with initial environment $x\in\RZ$ and where for all $n\in\mathbb{N}$, $Z_n(x)$ represents the size of the population at the $n$th generation.
To study this model, we use the probability generating function of the law of reproduction $\mu$, defined on $\RZ\times [0,1]$ by:
\begin{align*}
    \varphi(x,s) \defeq \sum_{k=0}^{+\infty} \mu_x(k)s^k.
\end{align*}
Moreover, for all $n\in\N$ and $x\in\RZ$, we consider the probability generating function of the random variable $Z_n(x)$, defined for all $s\in[0,1]$ by:
\begin{align*}
    \varphi^{(n)}(x,s) \defeq  \mathbb{E}[s^{Z_n(x)}].
\end{align*}
The probability of extinction $q:\RZ\to[0,1]$ is a function that associates to each $x\in\RZ$ the probability of extinction of the Galton-Watson process in dynamical environments with the initial environment $x$, i.e., with law of reproduction $(\mu_{T^nx})_{n\in\mathbb{N}}$. When $q(x)=1$, there is almost certain extinction of the process, and we say that $x\in\RZ$ is a bad environment.

We are therefore in the context of the model described in the first part of the introduction, and the probability of extinction $q$ is the smallest solution of the functional equation~\eqref{invgraph}, it is thus the invariant function of our model. If, for all ergodic probability measures, the survival probability is positive for $\nu$-almost all $x\in\RZ$, then we are in the uniformly supercritical case. In this setting, Theorem~\ref{thm2.2} shows that the H\"older regularity and the differentiability class of $q$ are controlled by the ratio $\frac{|\lambda_F|}{\lambda_u}$.

\begin{theorem}\label{thm2.2}
     Let $k\in\N$ and $\alpha\in(0,1]$. Assume:
     \begin{enumerate}[label=\alph*)]
     \item $k+\alpha <\frac{|\lambda_F|}{\lambda_u}$.
    \item $T$ is a $\mathcal{C}^{k+1}$ uniformly expanding transformation of $\RZ$.
    \item For all $\nu\in\mathcal{E}_T(\RZ)$, $\nu(q^{-1}(\{1\}))=0$ (uniformly supercritical case).
    \item $\varphi$ is $\mathcal{C}^{k+1}$ on $\RZ\times[0,1)$.
    \item $q(x)>0$ for all $x\in\RZ$.
\end{enumerate} Then $q$ is $\mathcal{C}^k$, and $q^{(k)}$ is $\alpha$-Hölder continuous.
\end{theorem}

In Theorem~\ref{thm2.2}, the case $k=0$ is given by \cite[Theorem 1.3.8]{morand2024galtonwatsonprocessesdynamicalenvironments}. We then proceed by induction on $k\in\mathbb{N}$. We know that the family of functions $\big(x\in\mathbb{\RZ}\mapsto\varphi^{(n)}(x,0)\big)_{n\in\mathbb{N}}$ converges uniformly to $q$ and is $(k+1)$-times differentiable. We then show that the sequence $\big(x\in\RZ\mapsto\partial_x^k\varphi^{(n)}(x,0)\big)_{n\in\mathbb{N}}$ converges uniformly to a function (whose Hölder regularity is controlled) which is therefore the $k$th derivative of $q$. To do this, we express \begin{align*}
    x\in\mathbb{\RZ}\mapsto\partial_x^k\varphi^{(n)}(x,0)
\end{align*} as a series of differentiable functions for which we control the uniform norm and the Hölder semi-norms as functions of the Lyapunov exponents.

  The Hölder regularity of invariant graphs in skew product systems has been studied in \cite{MR1605989,MR1677161} for invertible transformations. In \cite{MR1898800}, the authors investigate the $\mathcal{C}^k$ regularity for a hyperbolic dynamical system. In \cite{MR3816739}, the authors consider the case where the Lyapunov exponent in the fibre direction is zero on a set of periodic orbits.

\subsection*{Outline}

In Section~\ref{section1}, we present the model of Galton-Watson processes in dynamical environments. We also describe the main objects studied in this model (and give some of their properties): the probability generating function $\varphi$ and the probability of extinction $q$.

In Section~\ref{sect_effective}, we present techniques and algorithms for computing rigorous bounds on the invariant function $q$ and on the Lyapunov exponents $\lambda_F$ and $\lambda_u$. Subsection~\ref{s2.1} presents the lower bound on the Lyapunov exponents (Algorithm~\ref{alg_lb}), whereas Subsection~\ref{s2.2} presents the upper bound (Algorithm~\ref{alg_ub}). 

In Section~\ref{section3}, we provide an example of Galton-Watson process in dynamical environments to which we can apply the previous results. We analyze the numerical results and discuss the limitations of these algorithms as well as possible alternative methods.

Finally, Appendix~\ref{A} contains the proof of Theorem~\ref{thm2.2}.

\section{Model and previous results}\label{section1}

This section introduces the model of Galton-Watson processes in dynamical environments, the objects studied in this model (the probability generating function $\varphi$ and the probability of extinction $q$), and summarizes previous results in this setting.

\subsection{Galton-Watson processes in dynamical environments}
In this article, we consider $(\RZ,\mathcal{B}(\RZ),T)$ to be a topological discrete-time dynamical system, where $T$ is a $\mathcal{C}^1$ uniformly expanding transformation (there exists $\kappa>1$ such that for all $x\in\RZ$, $|T'(x)|\geq\kappa$).

We denote by $\mathcal{P}_T(\RZ)$ the set of $T$-invariant probability measures on $\RZ$ endowed with the weak*-topology, by $\mathcal{E}_T(\RZ)\defeq\{\nu\in\mathcal{P}_T(\RZ):\forall A \in\mathcal{B}(\RZ), T^{-1}A=A\implies\nu(A)\in\{0,1\}\}$ the set of $T$-ergodic probability measures on $\RZ$, and for all $k\in\N^*$, by $\text{Per}_k(\RZ)$ the set of $k$-periodic points in $\RZ$ (according to $T$).

We equip $\mathcal{P}(\mathbb{N})$ with the topology generated by the $\ell^1$ norm and its associated Borel algebra. Additionally, let:
\begin{align*}
    \mu:\left\{
    \begin{array}{lcl}
    \RZ&\to&\mathcal{P}(\mathbb{N})\\
    x&\mapsto&\mu_x
    \end{array}
    \right. .
\end{align*}
We assume that:
\begin{itemize}
    \item $x\in\RZ\mapsto \mu_x$ is continuous,
    \item $\mu_x\notin\{\delta_0,\delta_1\}$ for all $x\in\RZ$.
\end{itemize}

For each $x\in\RZ$, let the random process $(Z_n(x))_{n\in\mathbb{N}}$ be defined by:
\begin{align}
\left\{\begin{array}{ll}
   Z_0(x) &= 1 ,\\
    Z_{n+1}(x) &= \underset{k=1}{\overset{Z_n(x)}{\sum}}Y_{n,k}(x)\text{ for all }n\in\mathbb{N},\label{eq3}
    \end{array}\right. 
\end{align}
where $(Y_{n,k}(x))_{ (n,k) \in \mathbb{N}^2}$ is a family of independent random variables such that for all $(n,k) \in \mathbb{N}^2$, $Y_{n,k}(x)$ is distributed according to $\mu_{T^n x} $.
The family of random variables $(Z_n(x))_{n\in\mathbb{N}}$ is the Galton-Watson process in dynamical environments associated with the discrete-time dynamical system $(\RZ,\mathcal{B}(\RZ),T)$, the law of reproduction $\mu$, and the initial environment $x\in\RZ$.
\subsection{Extinction and generating function}

In dynamical environments, the law of reproduction depends on the environment, so the generating functions of the law of reproduction also depend on it.

\begin{definition}\label{def2}
The probability generating function of the law of reproduction $\mu$ is defined on $\RZ\times [0,1]$ as:
\begin{align*}
    \varphi(x,s) \defeq \sum_{k=0}^{+\infty} \mu_x(k)s^k = \mathbb{E}[s^{Y(x)}],
\end{align*}
where $Y(x)$ is distributed according to $\mu_{x} $.

For any $n \in \mathbb{N}$, the probability generating function of the distribution of $Z_n(x)$ is:
\begin{align*}
    \varphi^{(n)}(x,s) \defeq \sum_{k=0}^{+\infty} \mu_x^{(n)}(k)s^k= \mathbb{E}[s^{Z_n(x)}],
\end{align*}
where $x\in\RZ$, $s\in[0,1]$, and $\mu_x^{(n)}$ is the distribution of $Z_n(x)$.
\end{definition}

Proposition~\ref{prop5} relates the probability generating functions of the population at successive times.

\begin{prop}\cite[Proposition~1.2.4]{morand2024galtonwatsonprocessesdynamicalenvironments}\label{prop5}
For any $x \in \RZ$, $n,k\in \mathbb{N}$, and $s \in [0,1]$:
    \begin{align}\label{eq5}
 \varphi^{(n+k)}(x,s) = \varphi^{(k)}(x,\varphi^{(n)}(T^kx,s)).
    \end{align}
\end{prop}

The probability of extinction can be defined as a function of the probability generating functions.

\begin{definition}\label{def_q}
We define the probability of extinction $q$ for all $x\in\RZ$, by:\begin{align*}
    q(x)\defeq\underset{n\to\infty}{\lim}\nearrow \varphi^{(n)}(x,0)=\underset{n\to\infty}{\lim}\nearrow\mathbb{P}(Z_n(x)=0).
\end{align*}
\end{definition}
The following proposition is a consequence of Proposition~\ref{prop5} and of the continuity of $\varphi$.

\begin{prop}\cite[Proposition~1.2.6]{morand2024galtonwatsonprocessesdynamicalenvironments}\label{prop1}
For all $x\in\RZ$:
\begin{align}\label{e4}
    q(x)=\varphi(x,q(Tx)).
\end{align}
\end{prop}

A function such as $q$ or the constant function $1$ that satisfies Equation~\eqref{e4} is called an invariant function with respect to $(\varphi,T)$.

\begin{definition}
  For all $x\in\RZ$, let:
\begin{align*}
    m(x)\defeq \partial_s\varphi(x,1)\in (0,+\infty].
\end{align*}
\end{definition} 
For all $x\in\RZ$, $m(x)$ is the expected value of the probability distribution $\mu_x$.

\subsection{Regularity of the invariant function in the uniformly supercritical case}
The following theorem provides criteria for determining whether there is almost sure extinction of the process in the case where $\nu$ is an ergodic measure. It is a direct corollary of the results of Athreya and Karlin ~\cite[Corollary~1]{MR0298780} and ~\cite[Theorem~3]{MR0298780} in the context of our model.
\begin{theorem}\cite[Theorem 1.3.1]{morand2024galtonwatsonprocessesdynamicalenvironments}\label{thm1}
    Let $\nu\in\mathcal{E}_T(\RZ)$. Then\begin{align*}
        \nu(q^{-1}(\{1\}))=0 \text{ if and only if }\mathbb{E}_\nu[\log m(\cdot)]\defeq\int_{\rz}\log m(x)\,\mathrm{d}\nu(x)>0.
    \end{align*}
\end{theorem}
By definition, we say that the process is uniformly supercritical if the set $q^{-1}(\{1\})$ is of measure zero according to all ergodic probability measures. Thus, by Theorem~\ref{thm1} and the compactness of $\mathcal{P}_T(\RZ)$, the process is uniformly supercritical if and only if $\uo{\nu\in\mathcal{P}_T(\rz)}{}{\inf}  \mathbb{E}_\nu[\log m(\cdot)]>0$.

The question arises whether the regularity of $\varphi$ leads to regularity in the invariant function $q$ in the uniformly supercritical case. Under some integrability hypotheses, provided by (H\ref{hyp4}),  continuity is preserved by $q$ in the supercritical case \cite[Theorem 1.3.5]{morand2024galtonwatsonprocessesdynamicalenvironments}.\medskip

\encad{For $k\in\mathbb{N}$, let:
\begin{hyp}[H\ref{hyp4}(k)]\label{hyp4}~
\begin{enumerate}[label=\alph*)]
    \item $T$ is a $\mathcal{C}^{k+1}$ uniformly expanding transformation of $\RZ$.
    \item For all $\nu\in\mathcal{E}_T(\RZ)$, $\mathbb{E}_\nu[\log m(\cdot)]>0$ (uniformly supercritical case).\label{US}
    \item $\varphi$ is $\mathcal{C}^{k+1}$ on $\RZ\times[0,1)$.
    \item $q(x)>0$ for all $x\in\RZ$.
\end{enumerate} 
\end{hyp}}\medskip

Let $(H\ref{hyp4})\defeq (H\ref{hyp4}(0))$. For all $k\in\N$, we have that Hypothesis~(H\ref{hyp4}(k+1)) implies Hypothesis~(H\ref{hyp4}(k)). Let $k\in\N$. If for all $i\in\N$, $x\in\RZ\mapsto\mu_x(i)$ is $\mathcal{C}^{k+1}$, and $\lVert \mu_x(i)\rVert_{\mathcal{C}^{k+1}}=o(a^i)$ for all $a>1$, then Hypothesis~(H\ref{hyp4}(k))c) is satisfied.

The Lyapunov exponent $\lambda_u$ (Equation~\eqref{lu}) is well defined (and finite because $T$ is $\mathcal{C}^1$) by Fekete's subadditive lemma, and $\lambda_F$ (Equation~\eqref{lf}) is well defined by \cite[Lemma 4.2.2]{morand2024galtonwatsonprocessesdynamicalenvironments}. $\lambda_u$ controls the speed at which two orbits move apart. Moreover, $\lambda_u$ is bigger than the logarithm of the topological degree of $T$, so is positive. Conversely, the Lyapunov exponent in the fibre $\lambda_F$ allows us to control the speed at which the probability of extinction at the $n$th generation converges towards the invariant function $q$ \cite[Proposition~4.2.5]{morand2024galtonwatsonprocessesdynamicalenvironments}, and it is negative under (H\ref{hyp4}) by \cite[Lemma 4.2.6]{morand2024galtonwatsonprocessesdynamicalenvironments}. The ratio $\frac{|\lambda_F|}{\lambda_u}$ compares the convergence speed toward the invariant function with the rate at which two orbits separate under $T$. This ratio controls the Hölder seminorm of $\varphi^{(n)}$.

\begin{theorem}\cite[Theorem 1.3.8]{morand2024galtonwatsonprocessesdynamicalenvironments}\label{thm2.6}
   Let $\alpha\in(0,1]$. Assume (H\ref{hyp4}), and that $\alpha<\frac{|\lambda_F|}{\lambda_u}$. Then $q$ is $\alpha$-Hölder continuous.
\end{theorem}

In the setting of \cite{morand2024galtonwatsonprocessesdynamicalenvironments}, $T$ is a transformation on a compact space, and \cite[Theorem 1.3.8]{morand2024galtonwatsonprocessesdynamicalenvironments} provides sufficient conditions for $q$ to be Hölder continuous. In this article, the compact space under consideration is the smooth manifold $\RZ$. Then, we can ask what the differentiability class of $q$ is and what the Hölder regularity of its derivatives is. Theorem~\ref{thm2.2} answers this question by generalizing \cite[Theorem 1.3.8]{morand2024galtonwatsonprocessesdynamicalenvironments}. Its proof is in Appendix~\ref{A}.

\section{Effective estimates of the Lyapunov exponents}\label{sect_effective}

In this section, we present algorithms for computing rigorous bounds on the Lyapunov exponents $\lambda_F$ and $\lambda_u$ (see respectively Equations~\eqref{lf} and \eqref{lu} for their definitions) under (H\ref{hyp4}). Thanks to Theorem~\ref{thm2.2}, this will enable us to obtain rigorous bounds on the differentiability class and the Hölder regularity of $q$. 

\begin{remark}\label{rem1}
    In these algorithms, we also assume for simplicity that $T(0)=0$, and $T'>0$. However, the assumption $T(0)=0$ is not restrictive. Indeed, there exists a point $z\in\rz$ such that $T(z)=z$. Consider $\widetilde{T}=\tau_z\circ T \circ \tau_z^{-1}$ and $\widetilde{\mu}=\mu\circ \tau_z^{-1}$ with $\tau_z:x\in\rz\mapsto x-z\in\rz$, then $\widetilde{T}(0)=0$ and the Lyapunov exponents $\lambda_u$ and $\lambda_F$ of $(\widetilde{T},\widetilde{\mu})$ are the same as those of $(T,\mu)$.
\end{remark}

We introduce an auxiliary function $F$ to express the Lyapunov exponent in the fibre $\lambda_F$.

\begin{definition}\label{def1}
  For all \( x\in\RZ \), let:
\begin{align*}
    F(x)\defeq&\log \partial_s\varphi(x,q(Tx)).
\end{align*}
\end{definition}

\begin{lemma}\label{lmc}
    Under (H\ref{hyp4}), the functions $F$ and $\log |T'|$ are continuous.
\end{lemma}

\begin{proof}
    The function $x\in\RZ\mapsto F(x)=\log\partial_s\varphi(x,q(Tx))$ is continuous because $q$ is positive and continuous \cite[Theorem 1.3.5]{morand2024galtonwatsonprocessesdynamicalenvironments}, $\mu_x\neq\delta_0$ for all $x\in\RZ$, and $\varphi$ is continuous. Moreover, the function $\log |T'|$ is continuous because $T$ is $\mathcal{C}^1$ and for all $x\in\RZ$, $|T'(x)|>1$.
\end{proof}

Proposition~\ref{proplf} provides alternative definitions of the Lyapunov exponent in the fibre and of the Lyapunov exponent of the transformation $T$.

\begin{prop}\cite[Lemma~2.2.3, Proposition~4.2.3 and Lemma~4.2.4]{morand2024galtonwatsonprocessesdynamicalenvironments}\label{proplf}
    Under (H\ref{hyp4}), for all $0<a<1$,
    \begin{align*}
        \lambda_F&=\underset{n\to\infty}{\lim} \underset{x\in\rz}{\sup}\frac{1}{n}\sum_{k=0}^{n-1}F(T^kx)\\&=\underset{n\to\infty}{\lim} \underset{x\in\rz}{\sup}\frac{1}{n}\log \partial_s \varphi^{(n)}(x,a)\\&=\sup_{\nu\in\mathcal{P}_T(\rz)}\int_{\rz} F(x) \,\mathrm{d}\nu(x).
    \end{align*}
    Moreover,
    \begin{align*}
        \lambda_u&=\underset{n\to\infty}{\lim} \underset{x\in\rz}{\sup}\frac{1}{n} \sum_{k=0}^{n-1}\log|T'(T^kx)|\\&=\sup_{\nu\in\mathcal{P}_T(\rz)}\int_{\rz} \log|T'(x)| \,\mathrm{d}\nu(x).
    \end{align*}
\end{prop}

\begin{proof}
    To obtain the first equality for $\lambda_u$, apply the chain rule to differentiate $T^n$. The second equality for $\lambda_u$ follows directly from the semi-uniform ergodic theorem \cite[Theorem 1.9]{MR1734626}. The results for $\lambda_F$ are proven in \cite[Lemma~2.2.3, Proposition~4.2.3 and Lemma~4.2.4]{morand2024galtonwatsonprocessesdynamicalenvironments}.
\end{proof}

\subsection{Lower bound on the Lyapunov exponents}\label{s2.1}

This subsection presents a method for computing lower bounds on the Lyapunov exponents, $\lambda_F$ and $\lambda_u$.

\subsubsection{Lyapunov exponents on periodic orbits}

Proposition~\ref{prop2.1} shows that, to obtain the Lyapunov exponents $\lambda_F$ and $\lambda_u$, it is sufficient to consider the functions $F$ and $\log|T'|$ on periodic orbits. This follows from the density of ergodic measures supported on periodic orbits within the space of all ergodic measures \cite[Corollary~5]{MR718829}.

\begin{prop}\cite[Corollary~5]{MR718829}
\label{prop2.1}
    Assume (H\ref{hyp4}). For all continuous functions $f:\RZ\to \mathbb{R}$, \begin{align*}
        \sup_{\nu\in\mathcal{P}_T(\rz)}\int_{\rz} f(x) \,\mathrm{d}\nu(x)=\sup_{k\in\mathbb{N}^*}\sup_{x\in \text{Per}_k(\rz)}\frac{1}{k}\sum_{i=0}^{k-1}f(T^ix).
    \end{align*}
In particular, \begin{align*}
    \lambda_F=\sup_{k\in\mathbb{N}^*}\sup_{x\in \text{Per}_k(\rz)}\frac{1}{k}\sum_{i=0}^{k-1}F(T^ix) \quad \text{and} \quad \lambda_u=\sup_{k\in\mathbb{N}^*}\sup_{x\in \text{Per}_k(\rz)}\frac{1}{k}\sum_{i=0}^{k-1}\log|T'(T^ix)|.
\end{align*}
\end{prop}

In Proposition~\ref{prop2.1}, the equalities for $\lambda_F$ and $\lambda_u$ follow from the continuity of the functions $F$ and $\log |T'|$ (Lemma~\ref{lmc}), and from Proposition~\ref{proplf}.

\subsubsection{Effective lower bound on the Lyapunov exponents}\label{eff_low}

Thanks to the periodic orbits, we can numerically provide a lower bound on the Lyapunov exponents $\lambda_F$ and $\lambda_u$, as shown in Proposition~\ref{prop2.1}. Computing a lower bound for the Lyapunov exponents requires three steps: finding periodic points, computing a lower bound for $q$, and then computing a lower bound for the Lyapunov exponents.

\paragraph{Periodic point search}~

We compute an approximation of periodic orbits of $T$ of period less than some $M\in\N$ with an error smaller than $\varepsilon>0$. To achieve this, we search for all $k\leq M$, one representative of each primitive orbit of period $k$. 

We assume that $T(0)=0$ (see Remark~\ref{rem1}) and $T'>0$. Let $\widetilde{T}:\mathbb{R}\to\mathbb{R}$ be the lift of $T$ such that $\widetilde{T}(0)=0$. For $k\in\mathbb{N}^*$, the $k$-periodic points of $T$ are the points $x\in[0,1[$ such that $\widetilde{T}^k(x)-x\in\mathbb{N}$. If $d\geq 2$ is the topological degree of $T$, then there are $d^k-1$ periodic points of $T$ of period $k$ which are the unique $x_j\in[0,1[$ for $j\in\llbracket 0,d^k-2\rrbracket$ such that $\widetilde{T}^k(x_j)-x_j=j$. Using standard combinatorics, we choose one representative per primitive orbit of period $k$, and approximate it by dichotomy. We can then approximate the rest of the orbit by taking the images by $T$ of our first approximation (and use again dichotomy to obtain an error smaller than $\varepsilon$ because $T$ is expanding). For each $k$-periodic point $x$, we obtain a family of points $(y_i)_{i\in\llbracket 0,k-1\rrbracket}$ such that for all $i\in\llbracket 0,k-1\rrbracket$, $|T^i(x)-y_i|<\varepsilon$. For all $i\in\llbracket 0,k-1\rrbracket$, let $I_i=[y_i-\varepsilon,y_i+\varepsilon]$. Then, $(I_i)_{i\in\llbracket0,k-1\rrbracket}$ is a family of intervals of size $2\varepsilon$ such that for all $i\in\llbracket0,k-1\rrbracket$, $T^i(x)\in I_i$.

\paragraph{Lower bound on $q$}~

To compute a lower bound on $\lambda_F$, we must first compute a lower bound on $q \circ T$ along the periodic orbit (and thus of $q$) because $s\in[0,1]\mapsto \partial_s\varphi(x,s)$ is an increasing function for all $x\in\RZ$.

\begin{lemma}\label{lem_lb_q}
    Assume (H\ref{hyp4}) and that $T(0)=0$. Let $k\in\mathbb{N}^*$, $x\in\RZ$ be a $k$-periodic point of $T$, and $(I_i)_{i\in\llbracket0,k-1\rrbracket}$ be a family of intervals such that, for all $i\in\llbracket0,k-1\rrbracket$, $x_i\defeq T^i(x)\in I_i$. Then, for all $n\in\mathbb{N}$, \begin{align*}
    q(x)\geq \inf(\underbrace{\varphi(I_0,\varphi(I_1,\varphi(I_2,\ldots\varphi(I_{k-1},\varphi(I_0,\varphi(I_1, \ldots\varphi(I_{n-1[k]}}_{n \text{ times}},0)\ldots),
\end{align*}
where $m[k]$ is the remainder of the Euclidean division of $m$ by $k$.
Moreover, for all $i\in\llbracket0,k-1\rrbracket$, if $q_i\leq q(x_i)$, then $\inf(\varphi(I_{i-1[k]},q_i))\leq q(x_{i-1[k]})$.
\end{lemma}

\begin{proof}
   By Definition~\ref{def_q}, $q(x)\geq \varphi^{(n)}(x,0)$. The conclusion of the first inequality follows by induction, thanks to Proposition~\ref{prop5} and the $k$ periodicity of $x$. The second inequality follows from Proposition~\ref{prop1}, and because for all $x\in\RZ$, $s\mapsto \varphi(x,s)$ is increasing.
\end{proof}

Lemma~\ref{lem_lb_q} gives by induction an effective method to obtain a lower bound on $q$ along a periodic orbit. Indeed, we first compute a lower bound on $q(x)$. Secondly, we compute a lower bound on $q(T^{k-1}x)$ because $T(T^{k-1}(x))=x$. Then, by decreasing induction on $i\in\llbracket1,k-2\rrbracket$, we compute a lower bound on $q(T^i x)$ .

\paragraph{Lower bound on the Lyapunov exponents}~

Let $M\in\N$. In the previous two paragraphs, we computed the periodic orbits of $T$ of size smaller than $M$ and obtained a lower bound for $q$ along these periodic orbits. By Proposition~\ref{prop2.1},\begin{align}\label{in_lf}
    \lambda_F\geq \sup_{m\leq M}\sup_{x\in \text{Per}_m(\rz)}\frac{1}{m}\sum_{i=0}^{m-1}\partial_s\varphi(T^ix,q(T^{i+1}x)).
\end{align}
For all $x\in\RZ$, $s\mapsto\partial_s\varphi(x,s)$ is increasing. This provides a lower bound for $\lambda_F$, as described in Algorithm~\ref{alg_lb}.

\begin{algorithm}[H]
\caption{Compute a lower bound on $\lambda_F$}
\label{alg_lb}
\begin{algorithmic}[1]
\Require $\varepsilon > 0$, integer $M \geq 0$
\Ensure  a lower bound on $\lambda_F$
\State $\mathcal{L}_{\text{periodic\_orbits}} \gets \textsc{FindPeriodicOrbits}(M, \varepsilon)$
\cmt{\textsc{FindPeriodicOrbits} return all periodic orbits of $T$ of length $\leq M$, approximated with error $\leq \varepsilon$}
\State $\text{bound} \gets -\infty$

\ForAll{$\text{orbit} \in \mathcal{L}_{\text{periodic\_orbits}}$}
    \State $\text{bound\_q} \gets \textsc{LowerBoundQ}(\text{orbit}, \varepsilon)$
    \cmt{\textsc{LowerBoundQ} computes a lower bound on $q(Tx)$ for all $x$ in the orbit (see Lemma~\ref{lem_lb_q})}
    \State $\text{orbit\_bound} \gets \uo{i=1}{\text{length}(\text{orbit})}{\sum} \inf \partial_s \varphi([\text{orbit}[i]-\varepsilon,\text{orbit}[i]+\varepsilon], \text{bound\_q}[i])$
    \cmt{\eqref{in_lf}}
    \State $\text{bound} \gets \max(\text{bound}, \text{orbit\_bound})$
\EndFor
\State  \Return $\text{bound}$
\end{algorithmic}
\end{algorithm}

Thanks to Proposition~\ref{prop2.1}, we can compute a lower bound on the Lyapunov exponent of the transformation $T$ using the same method. For all $M\in\N^*$, \begin{align*}
    \lambda_u\geq \sup_{m\leq M}\sup_{x\in \text{Per}_m(\rz)}\frac{1}{m}\sum_{i=0}^{m-1}\log|T'(T^ix)|.
\end{align*}
However, the Lyapunov exponent of the transformation $T$ does not depend on $q$, making it easier to compute.

\subsection{Upper bound on the Lyapunov exponents}\label{s2.2}

This subsection presents a useful method for computing upper bounds on the Lyapunov exponents $\lambda_F$ and $\lambda_u$. These exponents are approximated by finding upper bounds of the Birkhoff sums of the functions $F$ and $\log |T'|$ on small intervals and using the expressions given in Corollary~\ref{coro2.2}.

\subsubsection{An expression of the Lyapunov exponents for computing an upper bound}

Corollary~\ref{coro2.2} provides an alternative expression for the Lyapunov exponents $\lambda_F$ and $\lambda_u$, which is useful for computing an upper bound on these exponents.

\begin{corollary}\label{coro2.2}
 Assume (H\ref{hyp4}). For all continuous functions $f:\RZ\to \mathbb{R}$, \begin{align}\label{abc}
        \sup_{\nu\in\mathcal{P}_T(\rz)}\int_{\rz} f(x) \,\mathrm{d}\nu(x)=\sup_{x\in\rz}\inf_{k\in\mathbb{N}^*}\frac{1}{k}\sum_{i=0}^{k-1} f(T^ix).
    \end{align}
In particular,
\begin{align*}
\lambda_F=\sup_{x\in\rz}\inf_{k\in\mathbb{N}^*}\frac{1}{k}\sum_{i=0}^{k-1} F(T^ix) \quad \text{and} \quad \lambda_u=\sup_{x\in\rz}\inf_{k\in\mathbb{N}^*}\frac{1}{k}\sum_{i=0}^{k-1} \log|T'(T^ix)|.
\end{align*}
\end{corollary}

\begin{proof}Let $f:\RZ\to \mathbb{R}$ be a continuous function.
    By Proposition~\ref{prop2.1}, \begin{align*}
        \sup_{\nu\in\mathcal{P}_T(\rz)}\int_{\rz} f(x) \,\mathrm{d}\nu(x)=\sup_{k\in\mathbb{N}^*}\sup_{x\in \text{Per}_k(\rz)}\frac{1}{k}\sum_{i=0}^{k-1}f(T^ix). 
    \end{align*} 
    Let $k\in\mathbb{N}^*$ and $x\in \text{Per}_k(\rz)$. By Pliss's lemma \cite[Lemma 11.8]{MR889254}, there exists $j\in\mathbb{N}$ such that \begin{align*}
        \frac{1}{k}\sum_{i=0}^{k-1}f(T^ix)= \inf_{n\in\mathbb{N}^*}\frac{1}{n}\sum_{i=0}^{n-1}f(T^{i+j}x)\leq\sup_{y\in\rz}\inf_{n\in\mathbb{N}^*}\frac{1}{n}\sum_{i=0}^{n-1} f(T^iy).
    \end{align*}
    Having established that \begin{align*}
        \sup_{\nu\in\mathcal{P}_T(\rz)}\int_{\rz} f(x) \,\mathrm{d}\nu(x)\leq\sup_{x\in\rz}\inf_{k\in\mathbb{N}^*}\frac{1}{k}\sum_{i=0}^{k-1} f(T^ix),
    \end{align*}
    it remains to prove the converse inequality to obtain Equality~\eqref{abc}. It is a direct corollary of the semi-uniform ergodic theorem \cite[Theorem 1.9]{MR1734626}. The equalities on $\lambda_F$ and $\lambda_u$ follow from the continuity of the functions $F$ and $\log |T'|$ (see Lemma~\ref{lmc}), and from Proposition~\ref{proplf}.
\end{proof}

\subsubsection{Effective upper bound on the Lyapunov exponents}\label{eff}

Computing an upper bound for the Lyapunov exponent $\lambda_F$ requires three steps, which are described in the next three paragraphs. First, we find a constant $K<1$ such that $K\geq q(x)$ for all $x\in\RZ$. Second, we compute an upper bound on $q$ using the expression $\varphi^{(n)}(x,K)$ with $n$ large. Finally, we compute an upper bound for the Lyapunov exponents using the upper bound on $q$ and the expression for the Lyapunov exponent $\lambda_F$ given in Corollary~\ref{coro2.2}.

\paragraph{Upper bound on $K$}~

In the uniformly supercritical case, there exists a constant $K<1$ such that $K\geq q(x)$ for all $x\in\RZ$ \cite[Lemma~4.1.1]{morand2024galtonwatsonprocessesdynamicalenvironments}. Lemma~\ref{lem2.1} justifies using the function $x\mapsto \varphi^{(n)}(x,K)$ with $n$ large as an upper bound on $q$.

\begin{lemma}\label{lem2.1}
Assume (H\ref{hyp4}). Let $K<1$ be such that $K\geq q(x)$ for all $x\in\RZ$. Then:\begin{itemize}
    \item For all $n\in\N$ and $x\in\RZ$, $\varphi^{(n)}(x,K)\geq q(x)$.
    \item The sequence of functions $(x\mapsto\varphi^{(n)}(x,K))_{n\in\mathbb{N}}$ converges uniformly to $q$.
\end{itemize}
\end{lemma}

\begin{proof}As $K<1$, the uniform convergence of this sequence of functions is proved by \cite[Corollary~4.1.2]{morand2024galtonwatsonprocessesdynamicalenvironments}.
Let $x\in\RZ$ and $n\in\mathbb{N}$.
As $s\mapsto \varphi^{(n)}(x,s)$ is increasing and by Proposition \ref{prop1},
\begin{align*}
    \varphi^{(n)}(x,K)&\geq\varphi^{(n)}(x,q(T^nx))=q(x).\qedhere
\end{align*}
\end{proof}

We need to find a suitable $K<1$. If there exists $N\in\N^*$ such that, for all $x\in\RZ$, $\varphi^{(N)}(x,K)\leq K$, then $K$ is suitable. However, we may choose a value of $N$ depending on $x$ by Lemma~\ref{lem_k2}.

\begin{lemma}\label{lem_k2}
Assume $(H\ref{hyp4})$. Let $K<1$ be such that for all $x\in\RZ$, there exists $N_x\in\N^*$ such that $\varphi^{(N_x)}(x,K)\leq K$. Then, $q(x)\leq K$ for all $x\in\RZ$.
\end{lemma}

\begin{proof}
Let $x\in\RZ$. We define a sequence of positive integers $(n_i)_{i\in\N}$ by the following induction relation:
\begin{align*} 
\left\{\begin{array}{ll}
  & n_0=N_x, \\
  & n_{i+1}=N_{T^{n_0+\ldots+n_i}x} \text{ for all } i\in\N.
   \end{array}\right.
\end{align*}
Then for all $i\in\N$,
\begin{align*}
    \varphi^{(n_0+\ldots+n_i)}(x,K)&=\varphi^{(n_0)}(x,\varphi^{(n_1)}(T^{n_0}x,\ldots,\varphi^{(n_i)}(T^{n_0+\ldots+n_{i-1}}x,K)\ldots)\leq K, 
\end{align*}
by induction, thanks to the definition of the sequence $(n_i)_{i\in\N}$. By \cite[Corollary~4.1.2]{morand2024galtonwatsonprocessesdynamicalenvironments}, \begin{align*}
    \varphi^{(n_0+\ldots+n_i)}(x,K)\underset{i\to +\infty}{\longrightarrow} q(x).
\end{align*} So, $q(x)\leq K$.
\end{proof}

\begin{remark}\label{rem2}
    In practice, it is sometimes difficult to check if a process is in the uniformly supercritical regime. However, Lemma~\ref{lem_k2} gives a sufficient condition for this to be the case.
\end{remark}

By Lemma~\ref{lem_k2}, we can certify whether a constant $0<C<1$ is an upper bound on $q$. This method is used in Algorithm~\ref{alg_K}. The algorithm returns true if it is verified that $C$ is an upper bound on $q$. It returns false if it is not verified (note that this does not necessarily imply that $C$ is not an upper bound of $q$).

\begin{algorithm}[H]
\caption{Verify whether $C$ is an upper bound on $\sup_{x\in\mathbb{R}} q(x)$}
\label{alg_K}
\begin{algorithmic}[1]
\Require $0 < C < 1$, integer $n_{\max} \geq 0$
\Ensure \texttt{True} if it is verified that $C \geq \sup_{x \in \mathbb{R}} q(x)$, \texttt{False} otherwise

\State Let $\mathcal{L}_{\text{interval}} \gets \{ [0, 1] \}$
\For{$i = 1$ to $n_{\max}$}
    \State Let $\mathcal{L}_{\text{new}} \gets$ empty list
    \cmt{$\mathcal{L}_{\text{new}}$ is the list of intervals after an iteration}
    \ForAll{$I \in \mathcal{L}_{\text{interval}}$}
        \If{for all $j = 1$ to $\left\lfloor \frac{i}{2} \right\rfloor + 1$, we have $\sup \varphi^{(j)}(I, C) > C$}
            \State Let $I_1, I_2 \gets \textsc{Split}(I)$
            \cmt{\textsc{Split} divides an interval into two intervals of the same size}
            \State Append $I_1$ and $I_2$ to $\mathcal{L}_{\text{new}}$
        \EndIf
    \EndFor
    \State $\mathcal{L}_{\text{interval}} \gets \mathcal{L}_{\text{new}}$
    \If{$\mathcal{L}_{\text{interval}}$ is empty}
        \State \Return \texttt{True}
    \EndIf
\EndFor
\State \Return \texttt{False}
\end{algorithmic}
\end{algorithm}

We need to identify a good candidate $C$ to be an upper bound on $q$. We estimate $\underset{x\in\rz}{\sup}q(x)$ by considering $C_{N,M,\varepsilon}\defeq\underset{i\in\llbracket0,2^N\rrbracket}{\sup}\varphi^{(M)}\left(\frac{i}{2^N},0\right)+\varepsilon$ with $\varepsilon>0$ and $M,N\in\N$. This is an upper bound on $q$ if $M$ and $N$ are large enough (with $\varepsilon>0$ fixed), because $q$ is continuous \cite[Theorem 1.3.5]{morand2024galtonwatsonprocessesdynamicalenvironments} and $\big(x\mapsto\varphi^{(n)}(x,0)\big)_{n\in\N}$ converges uniformly to $q$ \cite[Corollary~4.1.2]{morand2024galtonwatsonprocessesdynamicalenvironments}. Thanks to Algorithm~\ref{alg_K}, we check if $C_{N,M,\varepsilon}\defeq\underset{i\in\llbracket0,2^N\rrbracket}{\sup}\varphi^{(M)}\left(\frac{i}{2^N},0\right)+\varepsilon$ is an upper bound on $\underset{x\in\rz}{\sup}q(x)$. 
If $C_{N,M,\varepsilon}$ is not suitable, we will try with a new constant $\widetilde{C}\defeq C_{N,M,\varepsilon} + (1-C_{N,M,\varepsilon})\times\delta$ with $\delta>0$ (and this can be iterated if necessary). In practice, we can verify the condition: $\sup \varphi^{(N)}(I,\widetilde{C})\leq \widetilde{C}$ only on those intervals that did not satisfy it with the previous constant (because $\varphi^{(N)}$ is convex with respect to the second variable and $\varphi^{(N)}(x,1)=1$ for all $x\in\RZ$).

\paragraph{Upper bound on $q$}~

Let $0<K<1$ be an upper bound on $q$. By Lemma~\ref{lem2.1}, $x\mapsto\varphi^{(n)}(x,K)$ is an upper bound on $q$. In practice, we require an upper bound on $q$ which is piecewise constant over small intervals (see Inequality~\eqref{in_ub}). Let $I$ be an interval. For all $x\in I$, \begin{align}\label{in_uq}
    q(x)\leq \sup q(I)\leq \sup\varphi^{(n)}(I,K)=\sup\varphi(I,\varphi(TI,\ldots,\varphi(T^{n-1}I,K)\ldots).
\end{align}
Even if $I$ is small, if $j\in\N$ is large, then $T^{j}I$ can be large, meaning that the estimate \eqref{in_uq} may become less accurate as $n$ increases. To avoid this problem, we replace the function $\varphi$ by $\varphi_K:(x,s)\mapsto\min(\varphi(x,s),K)$ in \eqref{in_uq} (we can then define $\varphi_K^{(n)}$ for all $n\in\N$ by induction with the same relation as for $\varphi$, see \eqref{eq5}). By Lemma~\ref{lem2.1}, this is still an upper bound on $q(x)$, and it decreases with respect to $n$, because for all $x\in\RZ$, \begin{align}\label{minphiK}
    \varphi_K^{(n)}(x,K)=\min_{0\leq i\leq n}\big(\varphi^{(i)}(x,K)\big).
\end{align}

\paragraph{Upper bound on the Lyapunov exponents}~

Let $n\in\N^*$, $\big(k_{\max}(j)\big)_{j\in\llbracket1,n\rrbracket}\in(\N^*)^{n}$, and $(I_j)_{j\in\llbracket1,n\rrbracket}$ be a collection of intervals such that $\bigcup_{j=1}^n I_j=\RZ$. Then, by Corollary~\ref{coro2.2}, \begin{align}\label{in_ub}
    \lambda_F\leq\max_{j\in\llbracket1,n\rrbracket}\inf_{k\leq k_{\max}(j)}\sup\frac{1}{k}\sum_{i=0}^{k-1} \partial_s\varphi(T^iI_j,q(T^{i+1}I_j)).
\end{align}

With Algorithm~\ref{alg_ub}, we compute an upper bound on $\lambda_F$ using Inequality~\eqref{in_ub}. Smaller and smaller intervals are considered to improve the accuracy of the bound. To achieve this, we split the intervals into two. Hence, we only have intervals of the form $\left[\frac{i}{2^N},\frac{i+1}{2^N}\right]$ with $N\in\N$ and $i\in\llbracket0,2^N-1\rrbracket$, which are of length of $2^{-N}$. We define the structure \textsc{StructureOfInterval} to manage these intervals. 

\begin{algorithm}[H]
\caption{Define a mutable structure to represent an interval}
\label{str}
\begin{algorithmic}[1]
\Statex \textbf{Structure} \textsc{StructureOfInterval}
\State \hspace{1em} \textbf{interval:} \texttt{interval} \cmt{The interval}
\State \hspace{1em} \textbf{length:} \texttt{int} \cmt{The integer $N$ such that the interval length is $2^{-N}$}
\State \hspace{1em} \textbf{is\_new:} \texttt{bool} \cmt{True if the interval was created in the latest iteration}
\State \hspace{1em} \textbf{bound:} \texttt{float} \cmt{The upper bound computed on this interval}
\end{algorithmic}
\end{algorithm}
At each iteration, we split into two a proportion $\delta$ of intervals on which the bound is the worst, and bounds are computed for the new intervals. After $n_{\text{iteration}}$ iterations, the worst bound is returned, which is an upper bound on $\lambda_F$.
In practice, for $k_{\max}$ in Equation~\eqref{in_ub}, we use $k_{\max}(I)=\left\lfloor \frac{N}{2} \right\rfloor + 1$, where $2^{-N}$ is the length of the interval $I$. This ensures that the iterates of the interval $I$ by $T$ remain small.

\begin{algorithm}[H]
\caption{Compute an upper bound on $\lambda_F$}
\label{alg_ub}
\begin{algorithmic}[1]
\Require $0 < \delta < 1$, integers $n_{\text{iteration}} \geq 0$ and $n_K \geq 0$
\Ensure an upper bound on $\lambda_F$
\State Let $\mathcal{L}_{\text{interval}} \gets \{ ([0, 1], 0, \texttt{True}, +\infty) \}$
\cmt{$\mathcal{L}_{\text{interval}}$ is a list of \textsc{StructureOfInterval} (see Algorithm~\ref{str})}
\State Let $K \gets \textsc{FindK}(n_K)$
\cmt{See Algorithm~\ref{alg_K}}
\For{$n_{\text{iteration}}$ times}
    \State $\mathcal{L}_{\text{interval}} \gets \textsc{ComputeBounds}(\mathcal{L}_{\text{interval}}, K)$
    \cmt{\textsc{ComputeBounds} uses Equation~\eqref{in_ub} to compute bounds for new intervals}
    \State $\mathcal{L}_{\text{interval}} \gets \textsc{SplitWorstIntervals}(\mathcal{L}_{\text{interval}}, \delta)$
    \cmt{\textsc{SplitWorstIntervals} splits into two a proportion $\delta$ of intervals on which the bound is the worst}
\EndFor
\State \Return the maximum bound over all elements in $\mathcal{L}_{\text{interval}}$
\end{algorithmic}
\end{algorithm}

Thanks to Corollary~\ref{coro2.2}, we can compute an upper bound on the Lyapunov exponent of the transformation $T$ using the same method. Let $n\in\N^*$, $\big(k_{\max}(j)\big)_{j\in\llbracket1,n\rrbracket}\in(\N^*)^{n}$, and a $(I_j)_{j\in\llbracket1,n\rrbracket}$ be a collection of intervals such that $\bigcup_{j=1}^n I_j=\RZ$. Then,  \begin{align*}
    \lambda_u\leq\max_{j\in\llbracket1,n\rrbracket}\inf_{k\leq k_{\max}(j)}\sup\frac{1}{k}\sum_{i=0}^{k-1} \log|T'(T^iI_j)|.
\end{align*}

However, unlike in the expression of $F$, $q$ does not appear in the expression of $\log |T'|$. Therefore, it is unnecessary to compute an upper bound on $q$ in order to bound the Lyapunov exponent $\lambda_u$.

\section{Computation and analysis}\label{section3}
In this section, we apply the algorithms presented in Section~\ref{sect_effective} to different examples. This enables us to obtain rigorous bounds on the probability of extinction $q$ and on the Lyapunov exponents $\lambda_u$ and $\lambda_F$. We then analyse the results, discuss the limitations of these algorithms, and present alternative methods.
\subsection{The examples}

We present some examples to which we can apply the algorithms presented in Section~\ref{sect_effective}.

\begin{example}\label{ex1} Let:\begin{itemize}
    \item For all $N\in\mathbb{Z}$ and $\varepsilon\in\mathbb{R}$, $T_{N,\varepsilon}\defeq \left\{
    \begin{array}{lcl}\RZ&\to&\RZ\\
    x &\mapsto &N x+\varepsilon \sin(2\pi x) \text{ modulo } 1\end{array}\right.,$
    \item For all $\lambda\in\mathbb{R}$ and $\omega\in\RZ$, $\mu_{\lambda,\omega,.}\defeq\left\{
    \begin{array}{lcl}\RZ&\to&\mathcal{P}(\mathbb{N})\\
    x &\mapsto &\text{Pois}\left(e^{\lambda+\cos{(2\pi (x+\omega))}}\right)\end{array}
\right. $.\end{itemize}
Then, for all $x\in\RZ$ and $s\in[0,1]$, $\varphi_{\lambda,\omega}(x,s)=\exp{ \left(e^{\lambda+\cos{(2\pi (x+\omega))}}(s-1)\right) }$.
\end{example}

When $N=2$, $\varepsilon=0$, and $\omega=\frac{1}{2}$, we recover \cite[Example~1.2.2]{morand2024galtonwatsonprocessesdynamicalenvironments}.

 For all $N\geq 2$ and $\lambda,\omega,\varepsilon\in\mathbb{R}$ such that $|\varepsilon|<\frac{N-1}{2\pi}$ (so that $T_{
 N,\varepsilon}$ is expanding), Example~\ref{ex1} verifies Hypothesis~$H\ref{hyp4}(k)$ for all $k\in\mathbb{N}$, except perhaps Hypothesis~$H\ref{hyp4}(k)\ref{US}$ (uniformly supercritical case). Moreover, these examples satisfy the conditions that the transformation is increasing and sends $0$ to $0$, as assumed in Remark~\ref{rem1}.

 Let $N\geq 2$, $\omega\in\RZ$, and $\lambda,\varepsilon\in\mathbb{R}$ be such that $|\varepsilon|<\frac{N-1}{2\pi}$. To determine if we are in the uniformly supercritical case, we need to check if \begin{align}\label{cond}
     \uo{\nu\in\mathcal{P}_{T_{N,\varepsilon}}(\rz)}{}{\inf}  \mathbb{E}_\nu[\log m_{\lambda,\omega}(\cdot)]>0.
 \end{align}
 For all $x\in\RZ$, \begin{align*}
     \log m_{\lambda,\omega}(x)=\lambda+\cos(2\pi(x+\omega)).
 \end{align*}
 Thus, Inequality~\eqref{cond} can be expressed as:\begin{align}\label{thie}
     \lambda> -\uo{\nu\in\mathcal{P}_{T_{N,\varepsilon}}(\rz)}{}{\inf}  \mathbb{E}_\nu[\cos(2\pi(\cdot+\omega))].
 \end{align}
 In particular, \begin{itemize}
     \item if $\lambda>1$, we are in the uniformly supercritical case, and thus, Hypothesis~$H\ref{hyp4}(k)$ is satisfied for all $k\in\mathbb{N}$.
     \item if $N=2$ and $\varepsilon=0$, 
 Inequality~\eqref{thie} was studied in \cite{MR1785392}. The maximum of $\mathbb{E}_\nu[\cos 2\pi (\cdot+\omega)]$ is reached on a Sturm measure and is periodic for all $\omega\in\RZ$ outside a certain set of Hausdorff dimension zero. \cite[Annexe]{MR1785392} gives the corresponding maximising Sturm measure for some values of $\omega\in\RZ$.
 \end{itemize}

\subsection{Bounds on the probability of extinction}

In this subsection, we present the bounds obtained on the probability of extinction $q$ using arithmetic intervals. In addition, we explain why the control of $q$ by Fourier series is not appropriate in this context.

\subsubsection{Lower bound on the probability of extinction}

\begin{figure}[H]
    \centering
    \begin{subcaptionbox}{$\lambda=1,05$ and $\omega=0,5$\label{fig:image1}}[0.49\textwidth]
       {\includegraphics[width=\linewidth]{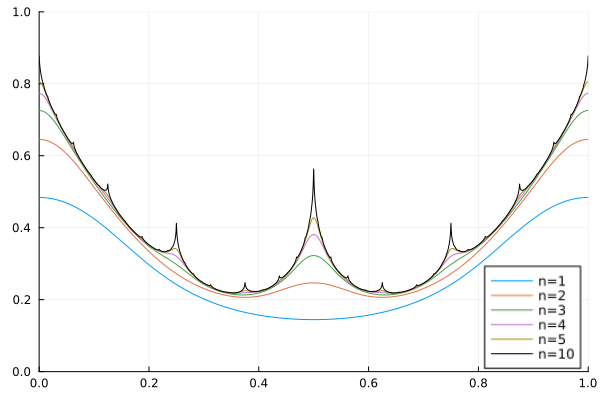}}
    \end{subcaptionbox}
    \hfill
    \begin{subcaptionbox}{$\lambda=0,55$ and $\omega=0$\label{fig:image2}}[0.49\textwidth]
        {\includegraphics[width=\linewidth]{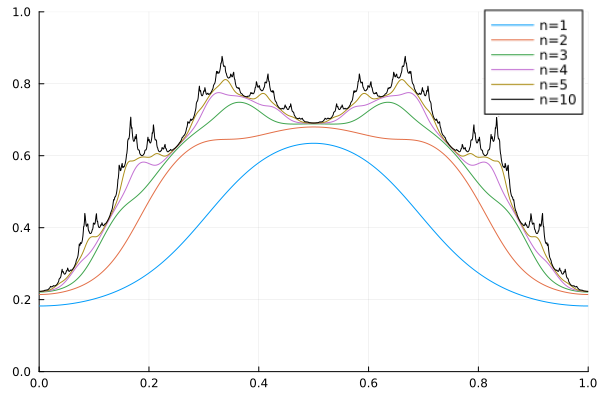}}
    \end{subcaptionbox}
    \caption{Plot of rigorous lower bounds on the function $x\in\RZ\mapsto\varphi^{(n)}_{\lambda,\omega}(x,0)$ in the case of Example~\ref{ex1} with $N=2$ and $\varepsilon=0$ for different values of $n$.}\label{fig1}
\end{figure}

As described in Sub-subsection~\ref{eff_low}, to obtain a lower bound on $q$, we use the fact that, by Definition~\ref{def_q}, for all $x\in\RZ$, $q(x)=\uo{n\to\infty}{}{\lim}\nearrow\varphi^{(n)}(x,0)$. Figure~\ref{fig1} shows the sequence of functions $\big(x\mapsto\varphi^{(n)}(x,0)\big)_{n\in\N}$ that converges increasingly to $q$ as $n$ increases.

\subsubsection{Upper bound on the probability of extinction}

Thanks to Equation~\eqref{in_uq}, we obtain an upper bound on $q$. For $0<K<1$ such that $K\geq q(x)$ for all $x\in\RZ$, the sequence of functions $\big(x\mapsto\varphi^{(n)}(x,K)\big)_{n\in\N}$ converges uniformly to $q$ and is everywhere greater than $q$.
Using the method described in Algorithm~\ref{alg_K}, we computed suitable values of $K$. We obtain $K=0,9386$ in the case of Figure~\ref{fig:image3}, and $K=0,9388$ in the case of Figure~\ref{fig:image4}.

\begin{figure}[H]
    \centering
   \begin{subcaptionbox}{$\lambda=1,05$ and $\omega=0,5$\label{fig:image3}}[0.49\textwidth]
        {\includegraphics[width=\linewidth]{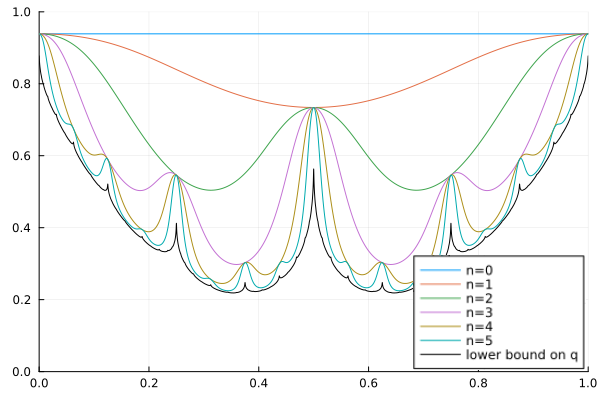}}
    \end{subcaptionbox}
    \hfill
    \begin{subcaptionbox}{$\lambda=0,55$ and $\omega=0$\label{fig:image4}}[0.49\textwidth]
        {\includegraphics[width=\linewidth]{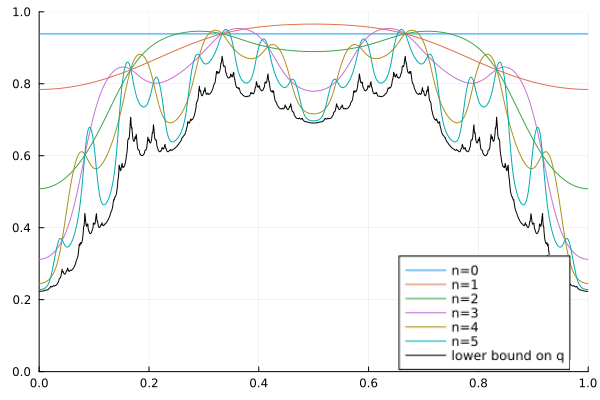}}
    \end{subcaptionbox}
    \caption{Plot of rigorous upper bounds on the function $x\in\RZ\mapsto\varphi^{(n)}_{\lambda,\omega}(x,K)$ in the case of Example~\ref{ex1} with $N=2$ and $\varepsilon=0$ for different values of $n\in\N$. The lower bound on $q$ is obtained by displaying the function $x\in\RZ\mapsto\varphi^{(10)}_{\lambda,\omega}(x,0)$.}
\end{figure}

However, the sequence $\big(x\mapsto\varphi^{(n)}(x,K)\big)_{n\in\N}$ is not necessarily decreasing (see for example Figure~\ref{fig:image4}). In Figure~\ref{fig:image3}, this sequence is decreasing. Since $\lambda=1,05>1$ (and $\omega=0,5$), for all $x\in\RZ$, the expectation of the law $\mu_{x,\lambda,\omega}$ is greater than $1$. Therefore, for all $x\in\RZ$, $\varphi_{\lambda,\omega}(x,K)\leq K$.

In practice, to improve the precision of the upper bound on $q$, we use: $\uo{0\leq i\leq n}{}{\min}\big(\varphi^{(i)}(x,K)\big)$ (instead of $\varphi^{(n)}(x,K)$), which decreases in $n$ to $q$ as described in Sub-subsection~\ref{eff}.

\subsubsection{Bounds on the probability of extinction using Fourier series}

    In the case of Example~\ref{ex1} with $\varepsilon=0$ (i.e. $T:x\mapsto Nx)$, we tried to approximate the probability of extinction using a Fourier series expansion.
    Let: \fonctionv{\widetilde{T}}{\mathcal{C}(\RZ,\RZ)}{\mathcal{C}(\RZ,\RZ)}{f}{f\circ T}
\fonctionv{\widetilde{\varphi}}{\mathcal{C}(\RZ,\RZ)}{\mathcal{C}(\RZ,\RZ)}{f}{(x\mapsto\varphi(x,f(x)))}
and $\widetilde{G}\defeq \widetilde{\varphi}\circ\widetilde{T}$. Then, for all $n\in\mathbb{N}$, $(x\mapsto \varphi^{(n)}(x,0))=\widetilde{G}^n(0)$, and so by Definition~\ref{def_q}, $q(x)=\uo{n\to\infty}{}{\lim}\widetilde{G}^n(0)$. In this case, if for all $x\in\RZ$,\begin{align*}
    f(x)=\sum_{n\in \mathbb{Z}}a_n c_n(x) \text{ with $c_n$ defined for all }n\in\N \text{ and }x\in\RZ \text{ by, } c_n(x)=e^{2i\pi n x}.
\end{align*} 
Then, for all $x\in\RZ$,
\begin{align*}
    \widetilde{T}(f)(x)=\sum_{n\in \mathbb{Z}}\widetilde{a}_n c_n(x) \text{ where } \widetilde{a}_n=\left\{\begin{array}{lcl}  a_{n/N} &\text{if}      &n\in N\mathbb{Z}\\
                                0    &\text{else}  &
\end{array} \right. .
\end{align*}
We can also characterise the action of $\widetilde{\varphi}$ on a function developed as a Fourier series, using the development of $x\mapsto\mu_x(k)$ as a Fourier series for all $k\in\N$. 

\begin{figure}[H]
    \centering
    \begin{minipage}{0.6\textwidth}
        \includegraphics[scale=0.34]{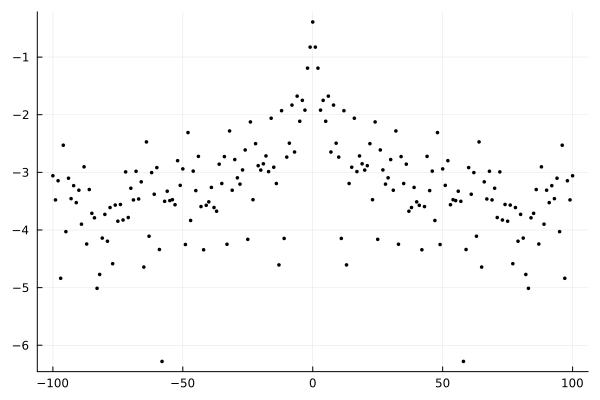}
    \end{minipage}%
    \begin{minipage}{0.4\textwidth}
       \captionof{figure}{Plot of the logarithm of the absolute value of the Fourier coefficients of an approximation of $q_\lambda$ for coefficients between $-100$ and $100$ in the case of Example~\ref{ex1} with $N=2$, $\varepsilon=0$, $\lambda=0.55$, and $\omega=0$.}
       \label{fi}
    \end{minipage}
\end{figure}

After applying $\widetilde{\varphi}$ (which is a nonlinear operator), there are no preferred Fourier frequencies that carry most of the information. This technique therefore loses its advantage, which would have been to keep only the dominant frequencies and control the negligible frequencies.

\subsection{Bounds on the Lyapunov exponents}

In this subsection, we present the bounds obtained on $\frac{|\lambda_F|}{\lambda_u}$. To compute a lower bound for this quantity, we use the upper bounds found on the Lyapunov exponents $\lambda_F$ and $\lambda_u$, and conversely for the upper bound.

\subsubsection{Periodic orbits}

\begin{figure}[H]
    \centering
    \begin{subcaptionbox}{$M\in\{1,2,3,4\}$}[0.49\textwidth]
        {\includegraphics[width=\linewidth]{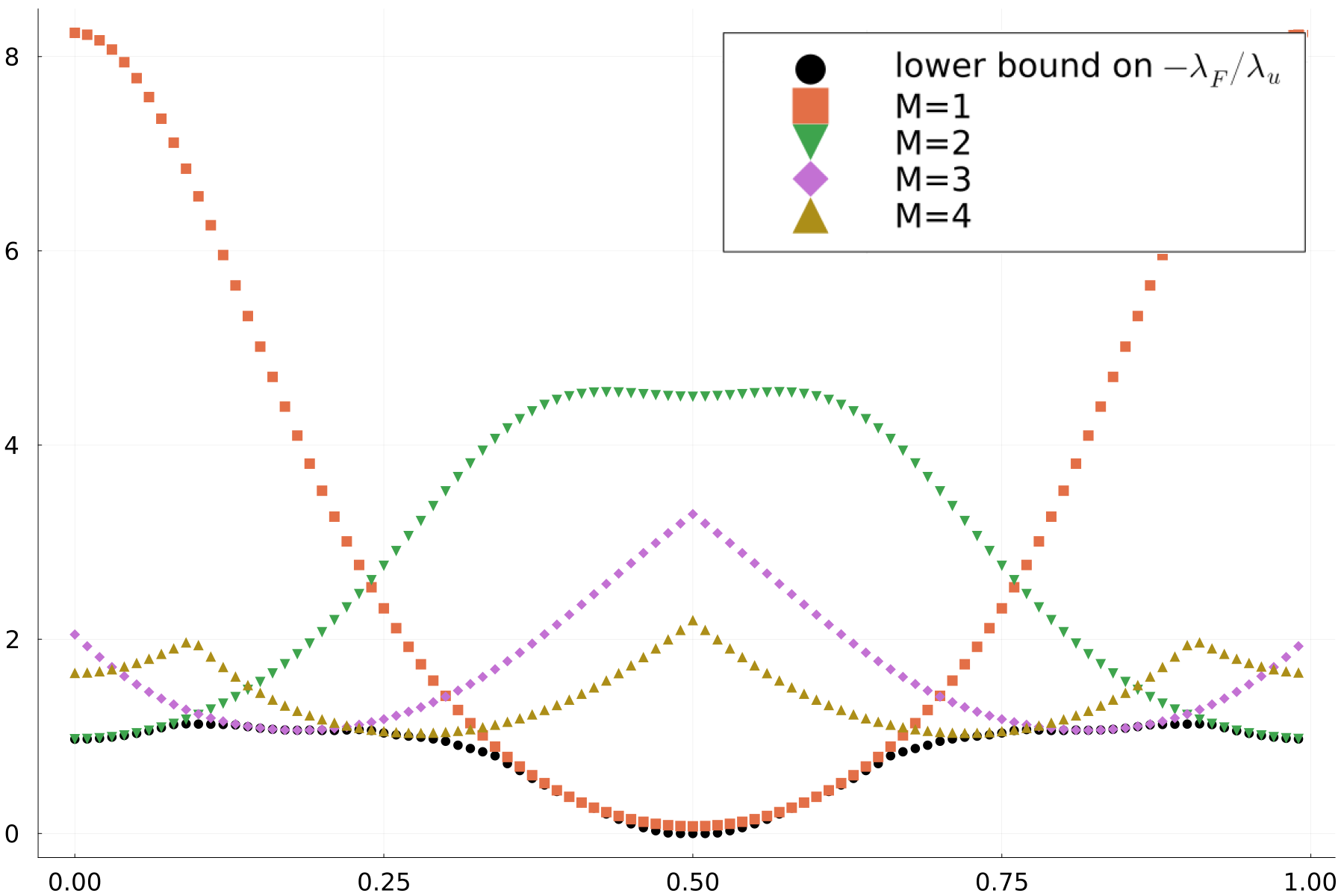}}
   \end{subcaptionbox}
    \hfill
    \begin{subcaptionbox}{$M\in\{5,8,11,14\}$}[0.49\textwidth]
       {\includegraphics[width=\linewidth]{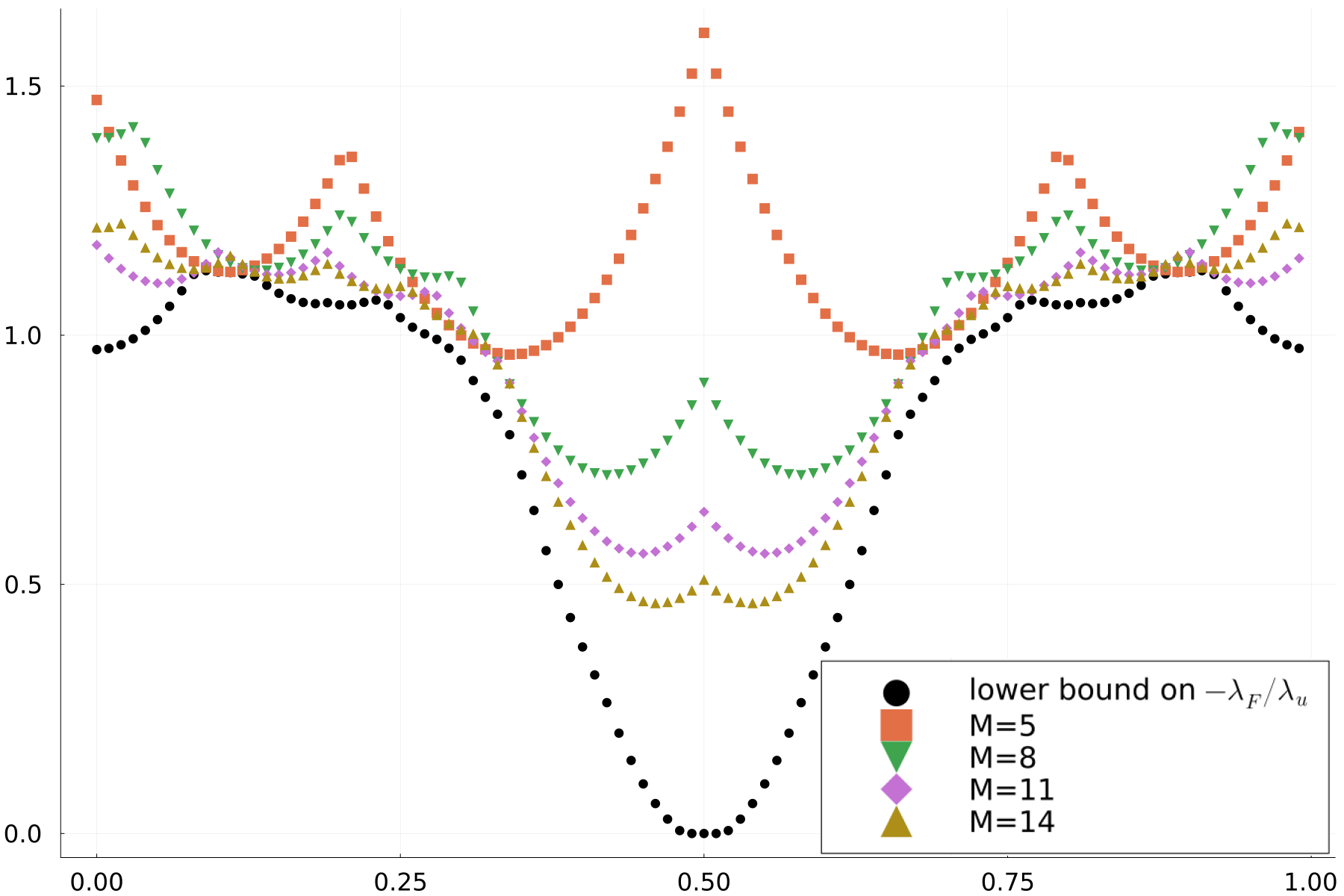}}
    \end{subcaptionbox}
    \caption{Plot of rigorous upper bounds on $\frac{|\lambda_F|}{\lambda_u}$ obtained by considering only orbits of primitive period $M$ as a function of $\omega\in\RZ$ in the case of Example~\ref{ex1} with $N=2$, $\varepsilon=0$, and $\lambda=1.05$. The lower bound on $\frac{|\lambda_F|}{\lambda_u}$ is obtained thanks to Algorithm~\ref{alg_ub}.}\label{fi1}
\end{figure}

Using Proposition~\ref{prop2.1}, we compute a lower bound on the Lyapunov exponents $\lambda_F$ and $\lambda_u$ (and thus an upper bound on $\frac{|\lambda_F|}{\lambda_u}$) by evaluating the functions $F$ and $\log|T'|$ along the periodic orbits. In Figure~\ref{fi1}, we plot an upper bound on $\frac{|\lambda_F|}{\lambda_u}$ obtained by considering only orbits of primitive period $M$ as a function of $\omega\in\RZ$ in the case of Example~\ref{ex1} with $N=2$, $\varepsilon=0$, and $\lambda=1.05$. In the case of Example~\ref{ex1}, the parameter $\lambda=1.05>1$ ensures that, for all $\omega\in\RZ$, the process is uniformly supercritical. As can be seen here, for some values of $\omega$, small periods suffice to obtain a good upper bound, whereas for others, larger orbits must be considered.  To obtain the bound given by Algorithm~\ref{alg_lb}, we must therefore consider the smallest bound obtained for all $M\leq \widetilde{M}$ with $\widetilde{M}\in\N$. Looking at the dependence on the variable $\omega$, the functions appear to be continuous and piecewise smooth. The points at which the function is not smooth correspond to a change in the orbit of primitive period $M$ which minimizes the function $F$.

\subsubsection{Accuracy of the bounds}

The accuracy of the bounds is defined as the difference between the upper bound and the lower bound. The difficulty lies in obtaining good bounds on $\lambda_F$ (compared to $\lambda_u$). The error is expected to come mainly from the upper bound. Indeed, to obtain a lower bound, we use the periodic orbits of the transformation $T$, and all the points in an orbit are approximated with an error smaller than $\varepsilon$. For the upper bound, we consider intervals of size $\varepsilon$, but applying $T$ increases their size since $T$ is expanding

Figure~\ref{fi2} plots the bounds on $\frac{|\lambda_F|}{\lambda_u}$ as a function of $\lambda\in\R$ in the case of Example~\ref{ex1} with $N=2$, $\varepsilon=0$, and $\omega=0$. The uniformly supercritical case occurs when $\lambda>0.5$. For all $x,\omega\in\RZ$, the family of probability measures $(\mu_{\lambda,\omega,x})_{\lambda\in\R}$ is stochastically increasing with respect to $\lambda\in\R$. Therefore, the probability of extinction $q_\lambda$ is decreasing in $\lambda$, so for all $x\in\RZ$, \begin{align}\label{Fex}
F_{\lambda}(x,q_{\lambda}(Tx))=\log\left(e^{\lambda+\cos{(2\pi (x+\omega))}}\exp{\left( e^{\lambda+\cos{(2\pi (x+\omega))}}(q_{\lambda}(Tx)-1) \right)}\right)
\end{align} 
is decreasing in $\lambda$, and thus $\frac{|\lambda_F|}{\lambda_u}$ is increasing with respect to $\lambda$.

\begin{figure}[H]
    \centering
    \begin{subcaptionbox}{$\lambda\in[0.5,1.0]$}[0.49\textwidth]
       {\includegraphics[width=\linewidth]{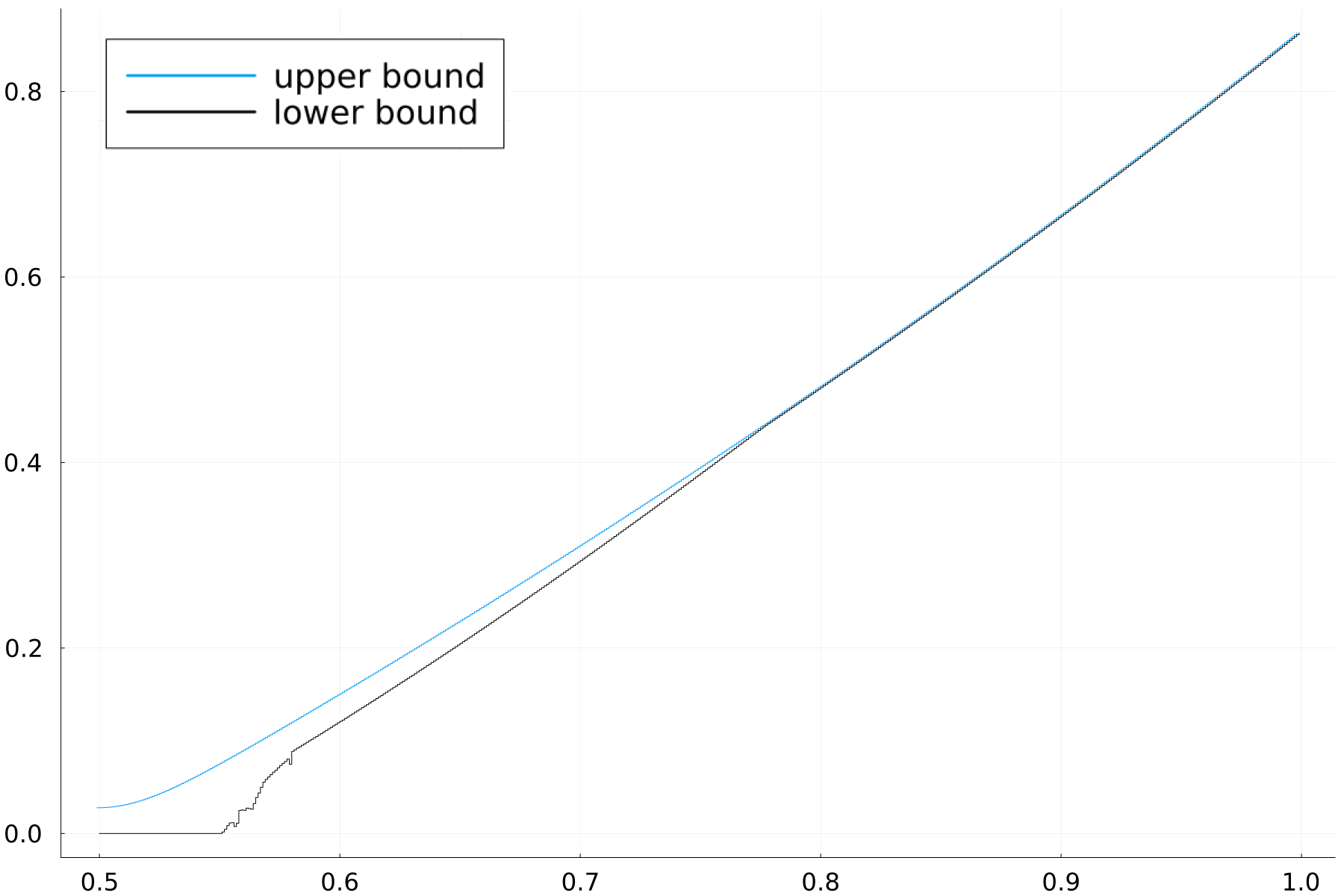}}
    \end{subcaptionbox}
    \hfill
    \begin{subcaptionbox}{$\lambda\in[0.5,5.0]$}[0.49\textwidth]
        {\includegraphics[width=\linewidth]{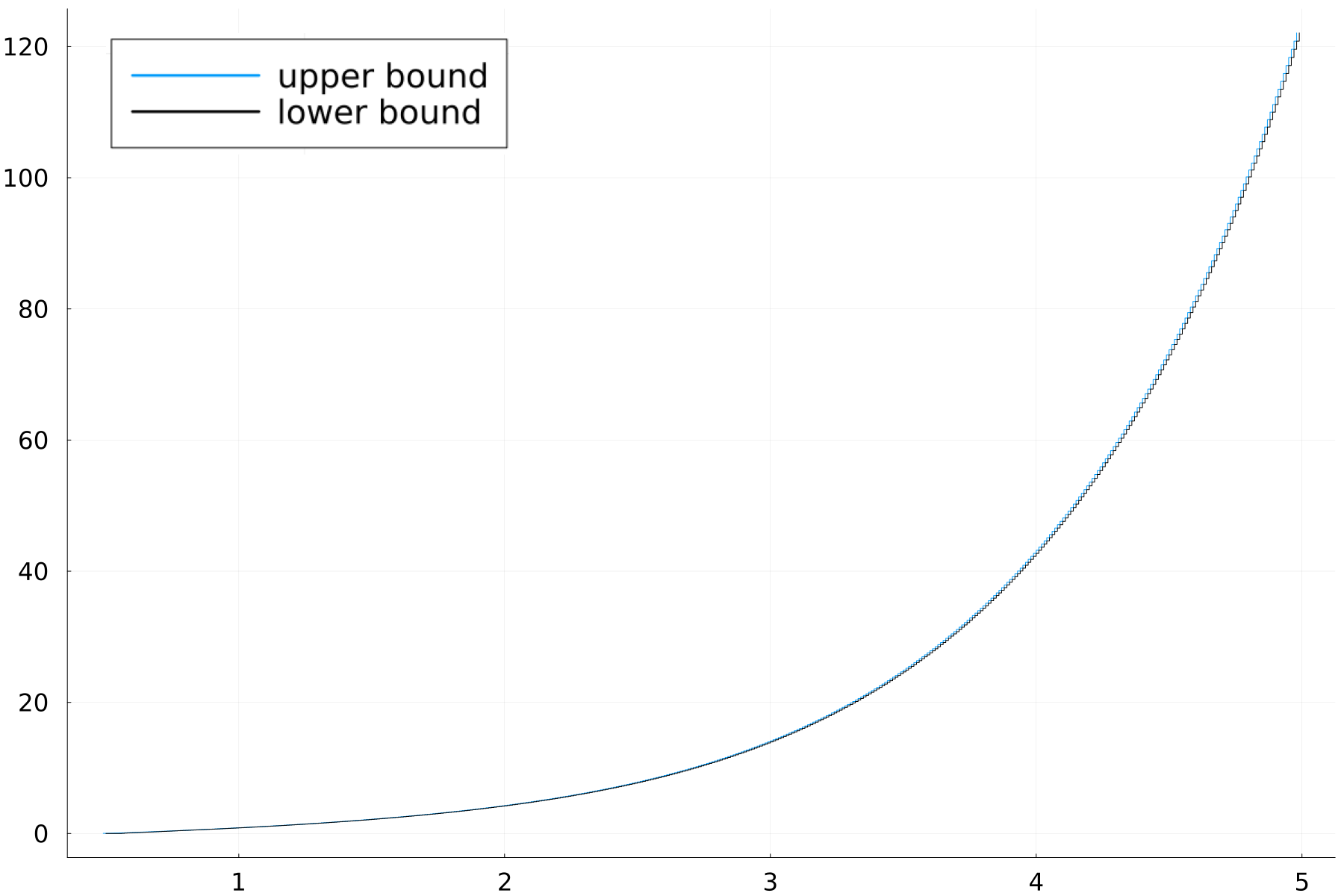}}
    \end{subcaptionbox}
    \caption{Plot of rigorous bounds on $\frac{|\lambda_F|}{\lambda_u}$ as a function of $\lambda\in\R$ in the case of Example~\ref{ex1} with $N=2$, $\varepsilon=0$, and $\omega=0$.}\label{fi2}
\end{figure}

The accuracy improves as $\lambda$ grows. This is because $\lambda_F$ depends on $q$, and the regularity of $q$ increases with $\lambda$. To obtain a good approximation when $q$ is not very regular, the function $q$ must be evaluated on small intervals.

The following table shows how the accuracy of Algorithm~\ref{alg_ub} improves when $n_{\text{iteration}}$ increases in the case of Example~\ref{ex1} with $N=2$, $\varepsilon=0$, $\omega=0$, and $\lambda=0.56$. Algorithm~\ref{alg_lb} gives an upper bound of $0.0886=8.86\times 100$ on $\frac{|\lambda_F|}{\lambda_u}$.
\begin{center}
 \begin{tabular}{|c|c|c|c|c|c|c|c|}
  \hline
  $n_{\text{iteration}}$ & 40  & 60  & 80 & 100 & 120 & 140 & 160\\[0.5ex]
  \hline
  Lower bound on $\frac{|\lambda_F|}{\lambda_u}$ $(\times 100)$ & 0 & 0.54 & 2.50 & 4.24 & 4.72 & 4.78 & 5.11 \\[0.5ex]
  \hline 
  Precision of the bounds on $\frac{|\lambda_F|}{\lambda_u}$ $(\times 100)$ & 8.86 & 8.32 & 6.36 & 4.62 & 4.14 & 4.08 & 3.75 \\[0.5ex] 
  \hline
\end{tabular}   
\end{center}
\begin{remark}
    For example, when $n_{\text{iteration}}=100$, the intervals considered in Algorithm~\ref{alg_ub} can be of size $2^{-100}$. We then need to ensure that the programme can handle such small intervals.
\end{remark}

 In the case of Example~\ref{ex1}, when $\lambda$ is large, the probability of extinction $q_\lambda$ converges exponentially fast to $0$. Thus, by using Equation~\eqref{Fex}, for all $x\in\RZ$,\begin{align*}
    F_{\lambda}(x,q_{\lambda}(Tx))\sim_{\lambda\to+\infty}-e^{\lambda}.
\end{align*}
Moreover, this equivalent is uniform in the variable $x\in\RZ$. Thus, by Proposition~\ref{proplf} (and as $\lambda_u$ does not depend on $\lambda$), \begin{align}\label{equi}
    \frac{\log(|\lambda_F|/\lambda_u)}{\lambda}\uo{\lambda\to+\infty}{}{\longrightarrow}1.
\end{align}
\begin{figure}[H]
    \centering
    \begin{minipage}{0.6\textwidth}
        \includegraphics[scale=0.14]{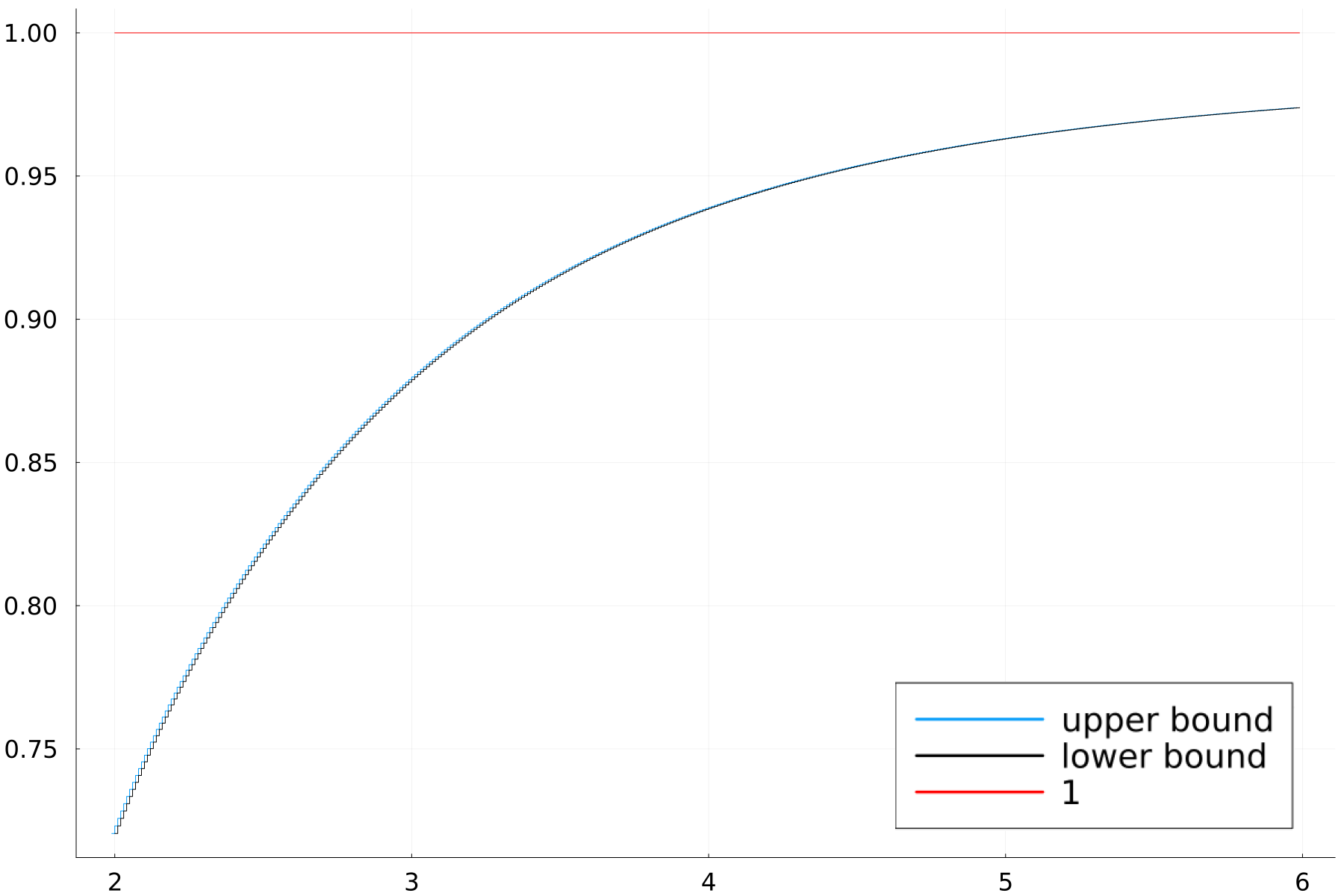}
    \end{minipage}%
    \begin{minipage}{0.4\textwidth}
       \captionof{figure}{Plot of rigorous bounds on $\frac{\log(|\lambda_F|/\lambda_u)}{\lambda}$ as a function of $\lambda\in[2,6]$ in the case of Example~\ref{ex1} with $N=2$, $\varepsilon=0$, and $\omega=0$.}
       \label{fi5}
    \end{minipage}
\end{figure}
In Figure~\ref{fi5}, Equivalence~\eqref{equi} is observed numerically in the case of Example~\ref{ex1} with $N=2$, $\varepsilon=0$, and $\omega=0$.

\begin{figure}[H]
    \centering
   \begin{subcaptionbox}{$N=3$, $\varepsilon=-0.1$, and $\omega=0.3$}[0.49\textwidth]
        {\includegraphics[width=\linewidth]{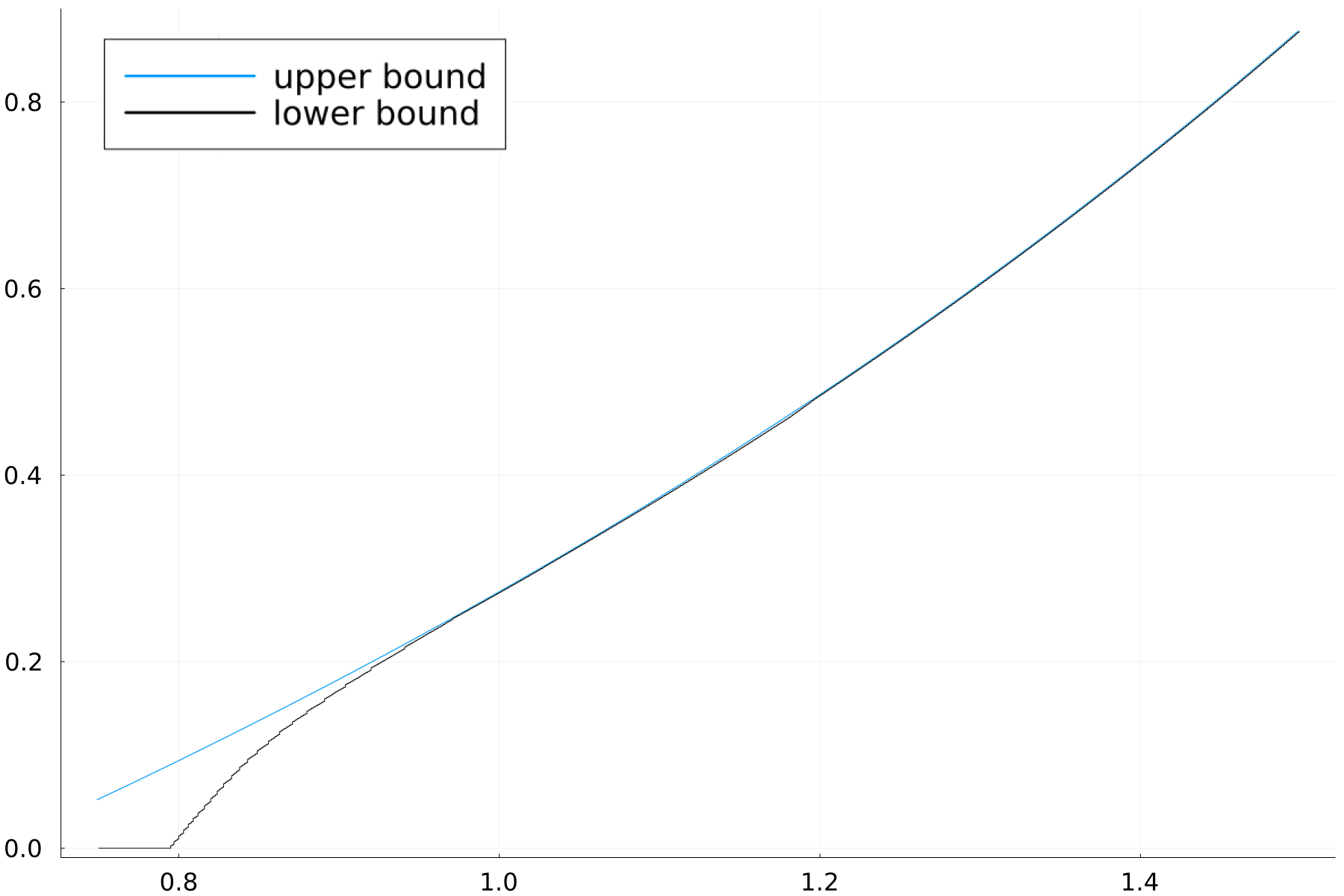}}
    \end{subcaptionbox}
    \hfill
    \begin{subcaptionbox}{$N=4$, $\varepsilon=0.2$, and $\omega=0.6$}[0.49\textwidth]
        {\includegraphics[width=\linewidth]{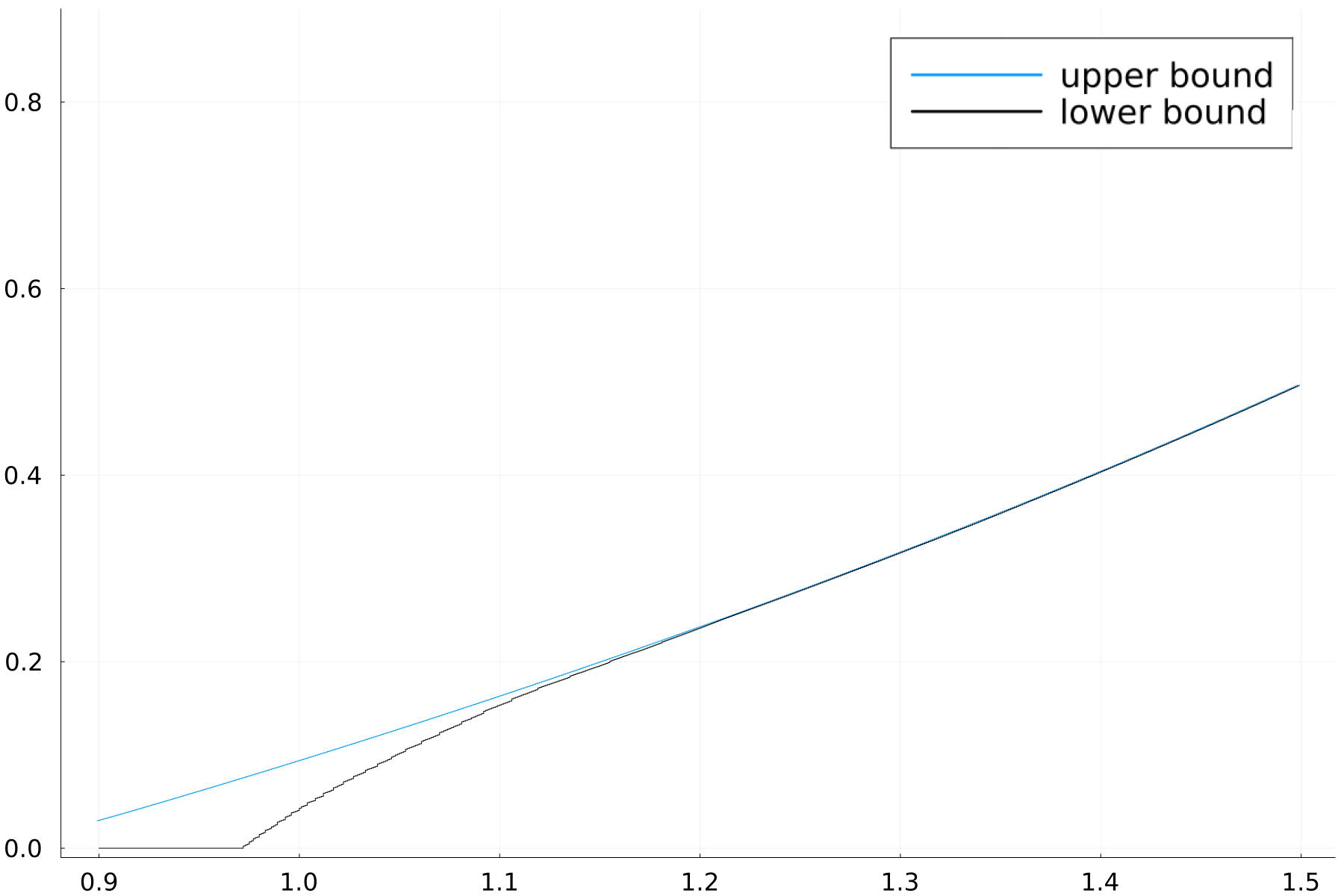}}
    \end{subcaptionbox}
    \caption{Plot of rigorous bounds on $\frac{|\lambda_F|}{\lambda_u}$ as a function of $\lambda\in\R$ in the case of Example~\ref{ex1}.}\label{fi3}
\end{figure}

In Figure~\ref{fi3}, we consider different laws of reproduction $\mu$ and transformations $T$ in the case of Example~\ref{ex1}. It appears that the uniformly supercritical case begins at different values of $\lambda$ (but there is no need to compute it explicitly thanks to Remark~\ref{rem2}). As in Figure~\ref{fi2}, we again observe an improvement in accuracy as $\lambda$ increases for the same reasons.

\subsubsection{Limitations and computational constraints}

To compute a lower bound on the Lyapunov exponents, we search for periodic orbits of the transformation $T$ (see Algorithm~\ref{alg_ub}). There are $d^k-1$ periodic points of period $k\in\N^*$ with $d\geq 2$, the topological degree of $T$. The number of periodic points of period $k$ increases exponentially as a function of $k$ (and the same is true for the number of primitive orbits of period $k$), so we cannot consider very large periodic orbits.

To compute an upper bound on the Lyapunov exponents, we evaluate the functions $F$ and $\log |T'|$ on increasingly smaller intervals. In Algorithm~\ref{alg_ub}, the intervals can be, at worst, of size  $2^{-n_{\text{iteration}}}$. When $n_{\text{iteration}}$ is large, it is necessary to ensure that the program can manage such small intervals. This negatively affects its speed. Furthermore, this level of precision is necessary to obtain accurate bounds, since the function $q$ may have low regularity ($\alpha$-Hölder continuous for an unknown $\alpha\in(0,1]$). 

\appendix
\section{Regularity of the invariant function}\label{A}

In this appendix, we prove Theorem~\ref{thm2.2}, which gives the differentiability class of $q$ and the Hölder regularity of its derivatives as a function of the ratio $\frac{|\lambda_F|}{\lambda_u}$.

\begin{remark}
     In this section, $T^{(n)}$ (respectively $q^{(n)}$) denotes the $n$-th derivative of $T$ (respectively of $q$) while $s\mapsto \varphi^{(n)}(x,s)$ is the probability generating function of the distribution of $Z_n(x)$ (see Definition~\ref{def2}).
\end{remark}

\subsection{Differentiability class of the invariant function}

In Lemma~\ref{lem2.3}, we express via induction the partial derivatives of $\varphi^{(n)}$ with respect to the variable $x$. It is expressed as a sum of a principal term controlled by induction and a polynomial in variables that we control. Polynomials that appear do not need to be explicit, but they can be computed by induction.
\begin{lemma}\label{lem2.3}
    Let $k\in\N^*$. Assume $(H\ref{hyp4}(k))$. Then, there exists a polynomial $P_k$ in the variables $T^{(z)}(x)$ for ${ z\leq k}$, $\dxsph{x}{\phn{Tx}{0}{n-1}}{z_1}{z_2}$ for $z_1+z_2\leq k$, and $\dxphn{Tx}{0}{n-1}{z}$ for $z<k$ such that for all $x\in\RZ$ and $n\geq 1$, \begin{align*}
        \dxphn{x}{0}{n}{k}=T'(x)^k\dxphn{Tx}{0}{n-1}{k}\dsph{x}{\phn{Tx}{0}{n-1}}+H_{k,n}(x),
    \end{align*} 
    where
    \begin{align*}
        H_{k,n}(x)= P_k[ (T^{(z)}(x))_{ z\leq k}, (\dxsph{x}{\phn{Tx}{0}{n-1}}{z_1}{z_2})_{z_1+z_2\leq k}, (\dxphn{Tx}{0}{n-1}{z})_{z<k}]. 
    \end{align*}
\end{lemma}

\begin{proof}
    We prove the result by induction on $k\in\N^*$. 
    For $k=1$, let $x\in\RZ$ and $n\in\N^*$.
    \begin{align*}
        \partial_x\big(\phn{x}{0}{n}\big)&=\partial_x\big(\ph{x}{\phn{Tx}{0}{n-1}}\big)\\
        &=T'(x)\dxphn{Tx}{0}{n-1}{}\dsph{x}{\phn{Tx}{0}{n-1}}+\dxph{x}{\phn{Tx}{0}{n-1}}{}\\
        &=T'(x)\dxphn{Tx}{0}{n-1}{}\dsph{x}{\phn{Tx}{0}{n-1}}+P_1[\dxph{x}{\phn{Tx}{0}{n-1}}{}]
    \end{align*}
    with $P_1[X]=X$.

    We suppose that the result is true at rank $k-1\in\N^*$. By the hypothesis of induction, for all $x\in\RZ$ and $n\geq 1$, \begin{align}\label{d3}
        &\dxphn{x}{0}{n}{k-1}=T'(x)^{k-1}\dxphn{Tx}{0}{n-1}{k-1}\dsph{x}{\phn{Tx}{0}{n-1}} \\&\hspace{0.5cm}+P_{k-1}[ (T^{(z)}(x))_{ z\leq k-1}, (\dxsph{x}{\phn{Tx}{0}{n-1}}{z_1}{z_2})_{z_1+z_2\leq k-1}, (\dxphn{Tx}{0}{n-1}{z})_{z<k-1}].\nonumber
    \end{align}
    Moreover, for all $z_1+z_2\leq k-1$, 
    \begin{align}\label{d1}
       &\partial_x\big(\dxsph{x}{\phn{Tx}{0}{n-1}}{z_1}{z_2}\big)\\&\hspace{0.5cm}=\dxsph{x}{\phn{Tx}{0}{n-1}}{z_1+1}{z_2}+T'(x)\dxphn{Tx}{0}{n-1}{}\dxsph{x}{\phn{Tx}{0}{n-1}}{z_1}{z_2+1}\nonumber,
    \end{align}
    and for all $z\leq k-1$, 
    \begin{align}\label{d2}
        \partial_x\big(\dxphn{Tx}{0}{n-1}{z}\big)=T'(x)\dxphn{Tx}{0}{n-1}{z+1}.
    \end{align}
    So, by Equations~\eqref{d1} and \eqref{d2}, there exists a polynomial $Q_k$ such that
    \begin{align}
        \partial_x\big(&T'(x)^{k-1}\dxphn{Tx}{0}{n-1}{k-1}\dsph{x}{\phn{Tx}{0}{n-1}}\big)\nonumber\\
        &=T'(x)^k\dxphn{Tx}{0}{n-1}{k}\dsph{x}{\phn{Tx}{0}{n-1}}\label{d4}
        \\&\hspace{1cm}+Q_k[(T^{(z)}(x))_{ z\leq k}, (\dxsph{x}{\phn{Tx}{0}{n-1}}{z_1}{z_2})_{z_1+z_2\leq k}, (\dxphn{Tx}{0}{n-1}{z})_{z<k}]\nonumber.
    \end{align}
Moreover, by Equations~\eqref{d1} and \eqref{d2}, there exists a polynomial $\widetilde{Q}_k$ such that 
\begin{align}\label{d5}
    &\partial_x\big(P_{k-1}[ (T^{(z)}(x))_{ z\leq k-1}, (\dxsph{x}{\phn{Tx}{0}{n-1}}{z_1}{z_2})_{z_1+z_2\leq k-1}, (\dxphn{Tx}{0}{n-1}{z})_{z<k-1}]\big)
    \nonumber\\&=\widetilde{Q}_k[(T^{(z)}(x))_{ z\leq k}, (\dxsph{x}{\phn{Tx}{0}{n-1}}{z_1}{z_2})_{z_1+z_2\leq k}, (\dxphn{Tx}{0}{n-1}{z})_{z<k}].
\end{align}
Therefore, by Equations~\eqref{d3}, \eqref{d4}, and \eqref{d5}, \begin{align*}
     \dxphn{x}{0}{n}{k}&=T'(x)^k\dxphn{Tx}{0}{n-1}{k}\dsph{x}{\phn{Tx}{0}{n-1}}\\&+P_k[(T^{(z)}(x))_{ z\leq k}, (\dxsph{x}{\phn{Tx}{0}{n-1}}{z_1}{z_2})_{z_1+z_2\leq k}, (\dxphn{Tx}{0}{n-1}{z})_{z<k}],
\end{align*}
where $P_k=Q_k+\widetilde{Q}_k$ is a polynomial.
\end{proof}

By induction, under the hypothesis of the preceding lemma, for all $x\in\RZ$ and $n\geq 1$,
\begin{align}
    \dxphn{x}{0}{n}{k}&=\sum_{\ell=0}^{n-1}\prod_{j=0}^{\ell-1}\bigg(T'(T^jx)^k\dsph{T^jx}{\phn{T^{j+1}x}{0}{n-1-j}}\bigg)H_{k,n-\ell}(T^\ell x)\nonumber\\
    &=\sum_{\ell=0}^{n-1}(T^\ell)'(x)^k\dsphn{x}{\phn{T^{\ell}x}{0}{n-\ell}}{\ell}H_{k,n-\ell}(T^\ell x).
\end{align}

Proposition~\ref{prop_dq} gives the differentiability class of $q$ using the expression for the partial derivative of $\varphi$ with respect to $x$ obtained in Lemma~\ref{lem2.3}.

\begin{prop}\label{prop_dq}
    Let $k\in\N^*$. Assume $(H\ref{hyp4}(k))$, and that $k<\frac{|\lambda_F|}{\lambda_u}$. Then $q$ is $\mathcal{C}^{k}$, and for all $x\in\RZ$, \begin{align*}
        q^{(k)}(x)=\sum_{\ell=0}^{+\infty}(T^\ell)'(x)^k\dsphn{ x}{q(T^{\ell}x)}{\ell}H_{k}(T^\ell x),
    \end{align*}
    where\begin{align*}
        H_k(x)= P_k[ (T^{(z)}(x))_{ z\leq k}, (\dxsph{x}{q(Tx)}{z_1}{z_2})_{z_1+z_2\leq k}, (q^{(z)}(Tx))_{z<k}].
    \end{align*}Moreover, $H_{k,n}$ converges uniformly to $H_k$ and $x\mapsto\dxphn{x}{0}{n}{k}$ converges uniformly to $q^{(k)}$ as $n\to+\infty$.
\end{prop}

\begin{proof}
    We show this result by strong induction on $k\in \N^*$.
    
    \underline{Basis case (k=1):}
    For all $x\in\RZ$ and $n\in\N^*$,\begin{align*}
        H_{1,n}(x)=\dxph{x}{\phn{Tx}{0}{n-1}}{},
    \end{align*}  and for  all $x\in\RZ$,
    \begin{align*}
        H_1(x)= \dxph{x}{q(Tx)}{}.
    \end{align*}
    $\partial_x\varphi$ is uniformly continuous on $\RZ\times[0,K]$ because $\varphi$ is $\mathcal{C}^1$ on $\RZ\times[0,K]$. Moreover, $x\mapsto\varphi^{(n)}(x,0)$ converges uniformly to $q$ by \cite[Corollary 4.1.2]{morand2024galtonwatsonprocessesdynamicalenvironments}. Thus, $H_{1,n}$ converges uniformly to $H_1$ when $n\to+\infty$.

 Let $\varepsilon>0$ be such that $(\lambda_u+\varepsilon)+\lambda_F+\varepsilon<0$. Let $N\in\N$ be such that for all $n\geq \ell\geq N$, \begin{align*}
       (T^\ell)'(x)\dsphn{ x}{\phn{T^{\ell}x}{0}{n-\ell}}{\ell} \leq (T^\ell)'(x)\dsphn{ x}{q(T^{\ell}x)}{\ell}\leq e^{\ell(\lambda_u+\varepsilon+\lambda_F+\varepsilon)}<1.
   \end{align*}
     As $(H_{1,n})_{n\in\N}$ converges uniformly to $H_1$ when $n\to+\infty$, $(H_{1,n})_{n\in\N}$ and $H_1$ are bounded by a common constant $C>0$.
     Define, for all $n,\ell\in\N$ and $x\in\RZ$,\begin{align*}
       f_{n,\ell}(x)\defeq\left\{\begin{array}{lll}
    (T^\ell)'(x)\dsphn{x}{\phn{T^{\ell}x}{0}{n-\ell}}{\ell}H_{1,n-\ell}(T^\ell x)&\text{if} &\ell<n\\
     0&\text{else} 
    \end{array}\right. ,
     \end{align*}
     and for all $\ell\in\N$ and $x\in\RZ$,
     \begin{align*}
         f_\ell(x)\defeq(T^\ell)'(x)\dsphn{ x}{q(T^{\ell}x)}{\ell}H_{1}(T^\ell x).
     \end{align*}
    So for all $\ell\in\N$, $f_{n,\ell}$ converges uniformly to $f_\ell$ when $n\to+\infty$.
    In addition, for all $\ell\in\N$ and $n\in\N$,
         \begin{align*}
             \ninf{f_{n,\ell}}\leq C e^{\ell(\lambda_u+\varepsilon+\lambda_F+\varepsilon)},
         \end{align*}
         which is summable (in $\ell$) and independent of $n$.
    Thus, by the dominated convergence theorem in the Banach space $\mathcal{C}(\RZ,\mathbb{R})$ \cite[Corollary~III.6.16]{MR117523},
     \begin{align*}
        \bigg( x \mapsto\dxphn{x}{0}{n}{}\bigg)=\sum_{\ell=0}^{+\infty}f_{n,\ell}\xrightarrow[n\to\infty]{\ninf{.}}\sum_{\ell=0}^{+\infty}f_{\ell}.
     \end{align*}
    Thus, $q$ is $\mathcal{C}^{1}$, for all $x\in\RZ$, \begin{align*}
        q'(x)=\sum_{\ell=0}^{+\infty}(T^\ell)'(x)\dsphn{ x}{q(T^{\ell}x)}{\ell}H_{1}(T^\ell x),
    \end{align*} and $x\mapsto\dxphn{x}{0}{n}{}$ converges uniformly to $q'$ when $n\to+\infty$.
    
    \underline{Inductive step:} Let $k\geq 2$. We suppose that the result is true for all $1\leq k'\leq k-1$, and we prove the result at rank $k$. For all $x\in\RZ$ and $n\in\N^*$,\begin{align*}
        H_{k,n}(x)=P_k[ (T^{(z)}(x))_{ z\leq k}, (\dxsph{x}{\phn{Tx}{0}{n-1}}{z_1}{z_2})_{z_1+z_2\leq k}, (\dxphn{Tx}{0}{n-1}{z})_{z<k}],
    \end{align*}  and for  all $x\in\RZ$,
    \begin{align*}
        H_k(x)= P_k[ (T^{(z)}(x))_{ z\leq k}, (\dxsph{x}{q(Tx)}{z_1}{z_2})_{z_1+z_2\leq k}, (q^{(z)}(Tx))_{z<k}].
    \end{align*}
    For all $z_1+z_2\leq k$, $\partial_x^{z_1}\partial_s^{z_2}\varphi$ is uniformly continuous on $\RZ\times[0,K]$ because $\varphi$ is $\mathcal{C}^k$ on $\RZ\times[0,K]$. Moreover, $x\mapsto\varphi^{(n)}(x,0)$ converges uniformly to $q$ by \cite[Corollary 4.1.2]{morand2024galtonwatsonprocessesdynamicalenvironments}. 
    So $x\mapsto \dxsph{x}{\phn{Tx}{0}{n-1}}{z_1}{z_2}$ converges uniformly to $x\mapsto \dxsph{x}{q(Tx)}{z_1}{z_2}$ when $n\to+\infty$ and for all $\ell\in\N$, $x\mapsto\dsphn{x}{\phn{T^{\ell}x}{0}{n-\ell}}{\ell}$ converges uniformly to $x\mapsto\dsphn{ x}{q(T^{\ell}x)}{\ell}$ when $n\to+\infty$. By the induction hypothesis, for all $1\leq z<k$, $x\mapsto \dxphn{Tx}{0}{n-1}{z}$ converges uniformly to $q^{(z)}\circ T$. Thus, $H_{k,n}$ converges uniformly to $H_k$ when $n\to+\infty$.

 Let $\varepsilon>0$ be such that $k(\lambda_u+\varepsilon)+\lambda_F+\varepsilon<0$. Let $N\in\N$ be such that for all $n\geq \ell\geq N$, \begin{align*}
       (T^\ell)'(x)^k\dsphn{ x}{\phn{T^{\ell}x}{0}{n-\ell}}{\ell} \leq (T^\ell)'(x)^k\dsphn{ x}{q(T^{\ell}x)}{\ell}\leq e^{\ell(k(\lambda_u+\varepsilon)+\lambda_F+\varepsilon)}<1.
   \end{align*}
     As $(H_{k,n})_{n\in\N}$ converges uniformly to $H_k$ when $n\to+\infty$, $(H_{k,n})_{n\in\N}$ and $H_k$ are bounded by a common constant $C>0$.
     Let us define, for all $n,\ell\in\N$ and $x\in\RZ$,\begin{align*}
       f_{n,\ell}(x)\defeq\left\{\begin{array}{lll}
    (T^\ell)'(x)^k\dsphn{x}{\phn{T^{\ell}x}{0}{n-\ell}}{\ell}H_{k,n-\ell}(T^\ell x)&\text{if} &\ell<n\\
     0&\text{else} 
    \end{array}\right. ,
     \end{align*}
     and for all $\ell\in\N$ and $x\in\RZ$,
     \begin{align*}
         f_\ell(x)\defeq(T^\ell)'(x)^k\dsphn{ x}{q(T^{\ell}x)}{\ell}H_{k}(T^\ell x).
     \end{align*}
    So for all $\ell\in\N$, $f_{n,\ell}$ converges uniformly to $f_\ell$ as $n\to+\infty$.
    In addition, for all $\ell\in\N$ and $n\in\N$,
         \begin{align*}
             \ninf{f_{n,\ell}}\leq C e^{\ell(k(\lambda_u+\varepsilon)+\lambda_F+\varepsilon)},
         \end{align*}
         which is summable (in $\ell$) and independent of $n$.
    Thus, by the dominated convergence theorem in the Banach space $\mathcal{C}(\RZ,\mathbb{R})$ \cite[Corollary~III.6.16]{MR117523},
     \begin{align*}
        \bigg( x \mapsto\dxphn{x}{0}{n}{k}\bigg)=\sum_{\ell=0}^{+\infty}f_{n,\ell}\xrightarrow[n\to\infty]{\ninf{.}}\sum_{\ell=0}^{+\infty}f_{\ell}.
     \end{align*}
    Therefore, $q$ is $\mathcal{C}^{k}$, for all $x\in\RZ$, \begin{align*}
        q^{(k)}(x)=\sum_{\ell=0}^{+\infty}(T^\ell)'(x)^k\dsphn{ x}{q(T^{\ell}x)}{\ell}H_{k}(T^\ell x),
    \end{align*} and $x\mapsto\dxphn{x}{0}{n}{k}$ converges uniformly to $q^{(k)}$ when $n\to+\infty$.
\end{proof}

\subsection{Hölder regularity of the invariant function}

After finding the differentiability class of $q$, we go further and ask what is the Hölder regularity of the highest derivative of $q$, which is what we do in Theorem~\ref{thm2.2}.

\begin{lemma}\label{lem2.4}
    Let $\alpha\in(0,1]$, $f:\RZ\to\mathbb{R}$ and $g:\RZ\to\RZ$ two $\mathcal{C}^1$ functions. Assume that there exist $a\in \R$, $b>0$ and $N\in\N$ such that for all $\ell\geq N$: \begin{align}\label{in10}
            \bninf{\prod_{j=0}^{\ell-1}f\circ g^j}\leq e^{a\ell}
        \quad \text{and} \quad
            \ninf{(g^\ell)'}\leq e^{b\ell}.
        \end{align}
    Then, there exists $C>0$, such that for all $\ell\in \N$, \begin{align*}
        \Big|\prod_{j=0}^{\ell-1}f\circ g^j\Big|_\alpha\leq Ce^{(a+\alpha b)\ell}.
    \end{align*}
\end{lemma}

\begin{proof}
    As for all $\ell\geq N$, \begin{align*}
            \bninf{\prod_{j=0}^{\ell-1}f\circ g^j}\leq e^{a\ell}.
            \end{align*}
        By Inequalities~\eqref{in10}, there exists $C_1,C_2>1$ such that for all $\ell\in \N$,\begin{align}\label{in9}
            \bninf{\prod_{j=0}^{\ell-1}f\circ g^j}\leq C_1 e^{a\ell} \quad \text{and} \quad
           \ninf{(g^\ell)'}=\bninf{\prod_{j=0}^{\ell-1}g'\circ g^j}\leq C_2e^{b\ell}.
        \end{align}
    Let $x,y\in\RZ$ and $\ell\in \N$. By Inequalities~\eqref{in9},
    \begin{align*}
        \Big|\prod_{j=0}^{\ell-1}f(g^jx)-\prod_{j=0}^{\ell-1}f(g^jy)\Big|&\leq \sum_{j=0}^{\ell-1}\bigg(\Big|\prod_{m=0}^{j-1}f(g^mx)\Big|\times\Big|f(g^jx)-f(g^jy)\Big|\times\Big|\prod_{m=j+1}^{\ell-1}f(g^my)\Big|\bigg)\\
        &\leq \sum_{j=0}^{\ell-1}C_1 e^{aj}|f(g^jx)-f(g^jy)|C_1  e^{a(\ell-j-1)}\\
        &\leq \sum_{j=0}^{\ell-1} \frac{C_1^2}{e^a}e^{a\ell}|f|_\alpha d(g^jx,g^jy)^\alpha\\
        &\leq \sum_{j=0}^{\ell-1} \frac{C_1^2}{e^a}e^{a\ell}|f|_\alpha  C_2 e^{\alpha bj} d(x,y)^\alpha\\
        &\leq Ce^{(a+\alpha b)\ell} \text{ with } C=\frac{C_1^2}{e^a}\frac{C_2}{1-e^{\alpha b}}|f|_\alpha \text{ because } b>0.\qedhere
    \end{align*}
\end{proof}

Proposition~\ref{prop_dq} and Lemma~\ref{lem2.4} allow us to prove Theorem~\ref{thm2.2}.

\begin{proof}[Proof of Theorem~\ref{thm2.2}]
    We show by strong induction on $k\in \N^*$, the hypothesis $\mathcal{P}(k)$: \textquotedblleft Assume $(H\ref{hyp4}(k))$. For all $\alpha\in(0,1]$ such that $k+\alpha <\frac{|\lambda_F|}{\lambda_u}$, the function $q^{(k)}$ is $\alpha$-Hölder continuous\textquotedblright. 
    
    For $k=0$, the result is true by \cite[Theorem 1.3.8]{morand2024galtonwatsonprocessesdynamicalenvironments}.

    Let $k\geq1$. We suppose that the result holds for all $k'\leq k-1$. Let $\alpha\in(0,1]$. Assume $(H\ref{hyp4}(k))$, and that $k+\alpha <\frac{|\lambda_F|}{\lambda_u}$. By Proposition~\ref{prop_dq}, for all $x\in\RZ$, 
    \begin{align*}
        q^{(k)}(x)=\sum_{\ell=0}^{+\infty}f_{\ell}(x),
    \end{align*}
    with\begin{align*}
        f_{\ell}(x)=(T^\ell)'(x)^k\dsphn{ x}{q(T^{\ell+1}x)}{\ell}H_{k}(T^\ell x).
    \end{align*}
Moreover, for all $x\in\RZ$, \begin{align*}
    H_k(x)= P_k[ (T^{(z)}(x))_{ z\leq k}, (\dxsph{x}{q(Tx)}{z_1}{z_2})_{z_1+z_2\leq k}, (q^{(z)}(Tx))_{z<k}]. 
\end{align*}
    As $T$ is $\mathcal{C}^{k+1}$, $\varphi$ is $\mathcal{C}^{k+1}$ on $\RZ\times[0,K]$, and by the induction hypothesis, \begin{itemize}
        \item $T^{(z)}$ is Lipschitz for $z\leq k$,
        \item $x\mapsto \dxsph{x}{q(Tx)}{z_1}{z_2}$ is Lipschitz for $z_1+z_2\leq k$,
        \item $q^{(z)}\circ T$ is Lipschitz for $z< k$.
    \end{itemize}
    Thus, $H_k$ is Lipschitz, so $\alpha$-Hölder continuous. Therefore, for all $\ell\in\N$, $f_\ell$ is $\alpha$-Hölder continuous.
    
    To prove that $q^{(k)}$ is $\alpha$-Hölder continuous, we need to control the $\alpha$-Hölder semi-norm of $f_\ell$ for $\ell\in\N$. 
    Let $\varepsilon>0$ be such that $(k+\alpha)(\lambda_u+\varepsilon)+\lambda_F+\varepsilon<0$. Let $N\in\N$ be such that for all $\ell\geq N$,
    \begin{align}\label{in1}
       \ninf{(T^\ell)'} \leq e^{\ell(\lambda_u+\varepsilon)},
    \end{align}
    and
    \begin{align}\label{in2}
        \ninf{x\mapsto\dsphn{ x}{q(T^{\ell}x)}{\ell}}\leq e^{\ell(\lambda_F+\varepsilon)}.
    \end{align}

Apply Lemma~\ref{lem2.4} twice: \begin{itemize}
    \item A first time with $f=T'$, $g=T$, $a=b=\lambda_u+\varepsilon>0$. Then by Inequality~\eqref{in1}, there exists $C_1>0$ such that for all $\ell\in \N$, \begin{align}\label{in3}
        |(T^\ell)'|_\alpha=\Big|\prod_{j=0}^{\ell-1} T' \circ T^j\Big|_\alpha\leq C_1 e^{(1+\alpha)(\lambda_u+\varepsilon)\ell}.
    \end{align} 
    \item A second time with $f:x\mapsto\partial_s\varphi(x,q(T x)$, $g=T$, $a=\lambda_F+\varepsilon$ and $b=\lambda_u+\varepsilon>0$. Then by Inequality~\eqref{in2}, there exists $C_2>0$ such that for all $\ell\in \N$, \begin{align}\label{in4}
       |x\mapsto \partial_s^{\ell}\varphi( x,q(T^{\ell} x))|_\alpha= \Big|x\mapsto\prod_{j=0}^{\ell-1} \partial_s\varphi(T^j x,q(T^{j+1} x)) \Big|_\alpha\leq C_2 e^{((\lambda_F+\varepsilon)+\alpha(\lambda_u+\varepsilon))\ell}.
    \end{align}
\end{itemize}
    Furthermore, by Inequality~\eqref{in1}, for all $\ell\geq N$,
    \begin{align}\label{in5}
        |H_k \circ T^\ell|_\alpha\leq |H_k|_\alpha \ninf{(T^\ell)'}^\alpha\leq |H_k|_\alpha e^{\ell\alpha(\lambda_u+\varepsilon)}.
    \end{align}

    Here is a table to summarise Inequalities~\eqref{in1},~\eqref{in2},~\eqref{in3},~\eqref{in4}, and~\eqref{in5}: for all $\ell\geq N$,
    \begin{center}
    \begin{tabular}{|c|c|c|}
    \hline
    bound on\ldots & the uniform norm & the $\alpha$-Hölder semi-norm  \\[0.5ex]
    \hline
    $(T^\ell)'$& $e^{\ell(\lambda_u+\varepsilon)}$ & $ C_1 e^{(1+\alpha)(\lambda_u+\varepsilon)\ell}$\\[0.5ex]
    \hline
    $x\mapsto \partial_s^{\ell}\varphi( x,q(T^{\ell} x))$& $e^{\ell(\lambda_F+\varepsilon)}$& $ C_2 e^{((\lambda_F+\varepsilon)+\alpha(\lambda_u+\varepsilon))\ell}$\\[0.5ex]
    \hline 
    $H_k \circ T^\ell$ & $ \ninf{H_k}$ & $e^{\ell\alpha(\lambda_u+\varepsilon)}|H_k|_\alpha $\\[0.5ex]
    \hline
 \end{tabular}.
\end{center}
 
    Thus, as a product of $\alpha$-Hölder functions, there exists $C>0$ such that for all $\ell\geq N$, \begin{align*}
        |f_\ell|_\alpha=|x\mapsto(T^\ell)'(x)^k\dsphn{ x}{q(T^{\ell}x)}{\ell}H_{k}(T^\ell x)|_\alpha\leq C e^{\ell((k+\alpha)(\lambda_u+\varepsilon)+(\lambda_F+\varepsilon))},
    \end{align*}
    which is summable because $(k+\alpha)(\lambda_u+\varepsilon)+\lambda_F+\varepsilon<0$.
    Thus $q^{(k)}=\uo{\ell=0}{+\infty}{\sum}f_\ell$ is $\alpha$-Hölder continuous. 
\end{proof}

 \section*{Acknowledgments}
The author would like to thank his thesis supervisor, Damien Thomine, for his advice and proofreading; Mateo Ghezal and Caroline Wormell, for their interesting discussions; Ewen Lallinec and Sylvain Faure, for their assistance with numerical computations.
\bibliographystyle{alpha}
\bibliography{biblio.bib}

\end{document}